\documentclass[11pt ,reqno]{amsart}

\usepackage{amsmath,amscd,amssymb,stmaryrd}
\usepackage{cite}
\usepackage{enumerate}
\usepackage{url}
\usepackage[colorlinks=true]{hyperref} 
\usepackage{todonotes,marginnote}  %julien added

\def\Vol{\mathrm{Vol}}

\def\tr{\mathrm{tr}}

\def\pr{\mathbb{P}}

\def\ddbar{\sqrt{-1}\partial\bar{\partial}}

\def\Id{\mathrm{Id}}

\def\Met{\mathrm{Met}}

\newtheorem{proposition}{Proposition}[section]
\newtheorem{lemma}[proposition]{Lemma}
\newtheorem{theorem}[proposition]{Theorem}
\newtheorem{corollary}[proposition]{Corollary}
\newtheorem{conjecture}{Conjecture}[section]

\theoremstyle{definition}
\newtheorem{remark}[proposition]{Remark}
\newtheorem{definition}[proposition]{Definition}

\def\C{\mathrm{C}}

\def\Vol{\mathrm{Vol}}
\def\ddbar{\sqrt{-1}\partial\bar{\partial}}

\newcommand{\R}{\mathbb{R}}

\newcommand{\Q}{\mathbb{Q}}
\newcommand{\comp}{\mathbb{C}}

\newcommand{\scK}{\mathcal{K}}
\newcommand{\scL}{\mathcal{L}}

\newcommand{\scI}{\mathcal{I}}
\newcommand{\scD}{\mathcal{D}}
\newcommand{\scH}{\mathcal{H}}

\newcommand{\scX}{\mathcal{X}}
\newcommand{\scO}{\mathcal{O}}
\newcommand{\scR}{\mathcal{R}}
\newcommand{\scB}{\mathcal{B}}
\newcommand{\scY}{\mathcal{Y}}

\DeclareMathOperator{\Bl}{Bl}

\DeclareMathOperator{\Supp}{Supp}

\DeclareMathOperator{\Grass}{Grass}
\DeclareMathOperator{\Hilb}{Hilb}
\DeclareMathOperator{\DF}{DF}

\DeclareMathOperator{\wt}{wt}
\DeclareMathOperator{\diag}{diag}

\usepackage{yfonts}  % for {\'i} in 

\author[R. Dervan]{Ruadha{\'i} Dervan}
\address{Ruadha{\'i} Dervan, DPMMS, Centre for Mathematical Sciences, Wilberforce Road, Cambridge CB3 0WB, United Kingdom.}
\email{R.Dervan@dpmms.cam.ac.uk}

\author[J. Keller]{Julien Keller}
\address{Julien Keller, Institut de Math\'ematiques de Marseille I2M, Aix-Marseille Universit\'e, 39, rue F. Joliot Curie, 13453 Marseille Cedex 13, France.}
\email{julien.keller@univ-amu.fr}

\allowdisplaybreaks[3]

\title{A finite dimensional approach to Donaldson's J-flow}

\begin{document}
\bibliographystyle{alpha}
%\begin{abstract}
%Consider a projective manifold and two distinct polarizations on this manifold. From this data and associated to a quantum parameter $k>>0$, we define  a flow over Bergman type metrics, denoted the J-balancing flow. We show that at the quantum limit $k\rightarrow + \infty$, the J-balancing flow converges towards the so-called J-flow of K\"ahler metrics. This flow was discovered by S.K. Donaldson and is related to the quest of constant scalar curvature K\"ahler metrics within the K\"ahler cone. For a fixed $k\gg 0$, the existence of a limit has a natural algebro-geometric interpretation in terms of a generalization of the notion Chow stability adapted to our setting. We relate this new algebro-geometric notion of stability to the notion of K-stability and the work of Lejmi-Sz{\'e}kelyhidi.
%\todo{complete abstract}
%\end{abstract}

\begin{abstract}

Consider a projective manifold with two distinct polarisations $L_1$ and $L_2$. From this data, Donaldson has defined a natural flow on the space of K\"ahler metrics in $c_1(L_1)$, called the J-flow. The existence of a critical point of this flow is closely related to the existence of a constant scalar curvature K\"ahler metric in $c_1(L_1)$ for certain polarisations $L_2$.

Associated to a quantum parameter $k\gg 0$, we define a flow over Bergman type metrics, which we call the J-balancing flow. We show that in the quantum limit $k\rightarrow + \infty$, the rescaled J-balancing flow converges towards the J-flow.  As corollaries, we obtain new proofs of uniqueness of critical points of the J-flow and also that these critical points achieve the absolute minimum of an associated energy functional. 

We show that the existence of a critical point of the J-flow implies the existence of J-balanced metrics for $k\gg 0$. Defining a notion of Chow stability for linear systems, we show that this in turn implies the linear system $|L_2|$ is asymptotically Chow stable. Asymptotic Chow stability of $|L_2|$ implies an analogue of K-semistability for the J-flow introduced by Lejmi-Sz{\'e}kelyhidi, which we call J-semistability. We prove also that J-stability holds automatically in a certain numerical cone around $L_2$, and that if $L_2$ is the canonical class of the manifold that J-semistability implies K-stability. Eventually, this leads to new K-stable polarisations of surfaces of general type. 
\end{abstract}

\maketitle

\tableofcontents

\section{Introduction}

Let $M$ be a smooth projective manifold equipped with two ample line bundles $L_1,L_2$ and of complex dimension $n\geq 2$. Fix Hermitian metrics $h_1\in {\rm Met}(L_1)$, $h_2\in {\rm Met}(L_2)$ such that the curvatures $\omega = c_1(h_1)$, and $\chi = c_1(h_2)$ are both K\"ahler forms. Given this data, Donaldson introduced a flow of K\"ahler metrics called the J-flow \cite{D9}. This flow is given by the following parabolic PDE in the (smooth) $\omega$-potentials $\phi_t$: \begin{equation}\label{Jflow}
\frac{\partial \phi_t}{\partial t}= \gamma- \frac{\chi \wedge (\omega+\ddbar \phi_t)^{n-1}}{(\omega+\ddbar\phi_t)^n},
\end{equation} where $\gamma$ is the topological constant $\frac{\int_M c_1(L_2)c_1(L_1)^{n-1}}{\int_Mc_1(L_1)^n}$ called the \textit{J-constant}. A \textit{critical metric} for Donaldson's J-flow is a solution at time $+\infty$ of the flow, i.e. a K\"ahler metric $\omega+\ddbar \phi\in c_1(L_1)$ solution of 
\begin{equation}\chi \wedge (\omega+\ddbar \phi)^{n-1} =\gamma (\omega+ \ddbar \phi)^n.\label{critical}
\end{equation}

Donaldson's motivation for studying this flow is its relation to an infinite dimensional moment map picture. Soon after its introduction, the work of Chen, together with a refinement due to Song-Weinkove, proved that if $L_2=K_M$ is the canonical class of $M$, then the existence of a critical point of the J-flow implies properness of the Mabuchi functional for $(M,L_1)$ \cite{C2,SW}. Properness of the Mabuchi functional is conjecturally equivalent to the existence of a constant scalar curvature K\"ahler (which we abbreviate to cscK) metric on $L_1$. Indeed, the existence of solution of the J-flow is equivalent to properness of a functional, which we denote $I_{\mu_J}$, due to work of Chen and Collins-Sz\'ekelyhidi \cite{C2, Sz-Co}. When $L_2=K_M$ the functional $I_{\mu_J}$ is a component of the Mabuchi functional, which explains the relation between the cscK problem and the J-flow. Since then there has been much work in providing numerical criteria for the existence of critical metrics of the J-flow, in order to better understand the existence of cscK metrics \cite{C1,We2,SW}, under the assumption $K_M$ is ample. Chen has also very recently proposed a program to prove the existence of cscK metrics using a continuity method, in which the metric at $t=0$ is a critical metric of the J-flow \cite{C3}. We later discuss these motivations for studying the J-flow in further detail. 

We now give an overview of the main results of this paper. As is now well understood, one can quantise metrics on $L_1$ using metrics on the finite dimensional vector spaces $H^0(M,L_1^k)$ with quantum parameter $k$, called Bergman metrics. In the present work we study a flow on the space of Bergman metrics, which we call the J-balancing flow. Critical points of the J-balancing flow are called J-balanced metrics, and these fit naturally into a finite dimensional moment map picture. Our main result is that in the quantum limit $k\to +\infty$, the J-balancing flow converges to the J-flow.
\begin{theorem} Fix $T>0$ and let $\omega_k(t)$ be the solution of the J-balancing flow, for $t\in  [0,T]$. Then as $k\to\infty$, the sequence $\omega_k(t)$ converges in $C^{\infty}$ to the solution of the J-flow as $k\to\infty$. Furthermore, the convergence is $\C^1$ in the variable $t$. Assuming there is a critical point of the J-flow, the convergence holds for all $t>0$. \end{theorem}

By showing J-balanced metrics are unique, we obtain the following Corollary, which is an analogue of Donaldson's quantisation proof of uniqueness of cscK metrics \cite{D1}. 

\begin{corollary} Critical metrics of the J-flow are unique. \end{corollary}

This was first proven by Chen using the strict convexity of the $I_{\mu_J}$ functional along certain geodesics in the space of K\"ahler metrics on $L_1$ \cite{C1}. Similarly, we recover the fact that J-balanced metrics achieve the absolute minimum of an associated functional.

\begin{corollary}\label{corabsmin} Critical metrics of the J-flow achieve the absolute minimum of the $I_{\mu_J}$ functional.\end{corollary} 

This is analogous to Donaldson's proof that cscK metrics achieve the absolute minimum of the Mabuchi functional \cite{D2}. 

Our next results relate the existence of a critical metric to algebro-geometric notions of stability. We define a notion of \emph{Chow stability} for a linear system on a polarised manifold, and by relating this notion of stability with the existence of J-balanced metrics we obtain the following.

\begin{theorem}\label{introchow} The existence of a critical metric of the J-flow implies the linear system $|L_2|$ is asymptotically Chow stable. \end{theorem}

This is an analogue of Donaldson's result proving the existence of a cscK metric implies asymptotic Chow stability \cite{D1}. The first result relating the J-flow to algebro-geometric stability is due to Lejmi-Sz\'ekelyhidi \cite{L-Sz}. They proved the existence of a critical point implies an analogue of K-stability in this case, which we call J-stability. Since asymptotic Chow stability of $|L_2|$ implies J-semistability, we obtain a new proof of the semistability part of their result. We remark that they then use a perturbation argument to prove strict J-stability.

\begin{corollary} If $(M,L_1,L_2)$ admits a critical metric of the J-flow, then it is J-semistable. \end{corollary}

The existence of a cscK metric is conjecturally equivalent to the algebro-geometric notion of K-stability; this is called the Yau-Tian-Donaldson conjecture. Through Chen's work one would therefore expect that when $L_2=K_M$, J-stability is naturally related to K-stability \cite{C2}. We show that this is indeed the case. We extend Lejmi-Sz\'ekelyhidi's definition of J-stability to the case where $L_2$ is an arbitrary (not necessarily ample) line bundle, and our results hold in this generality. 

\begin{theorem} If $(M,L_1,K_M)$ is J-semistable, then $(M,L_1)$ is K-stable. \end{theorem}

Similarly, analogously to the work giving numerical criteria for the existence of critical metrics of the J-flow, we give a numerical criterion for J-stability.

\begin{theorem} Suppose $\gamma L_1- L_2$ is nef, with $\gamma>0$. Then $(M,L_1,L_2)$ is J-stable. \end{theorem}

Here we have used additive notation for line bundles. Note that while $\gamma>0$ is automatic when $L_2$ is ample, our result holds for $L_2$ \emph{arbitrary} provided $\gamma>0$. Combining the previous two results gives the following, which was also proven by the first author using a different method \cite[Theorem 1.7]{Derv2}.

\begin{corollary} Let $\gamma$ be the J-constant for $(M,L_1,K_M)$, and suppose $\gamma>0$. If $\gamma L_1 - K_M$ is nef, then $(M,L_1)$ is K-stable. \end{corollary}

We emphasise that in contrast to the corresponding analytic results due to Chen, Weinkove and Song-Weinkove proving properness of the Mabuchi functional \cite{C1,We2,SW}, we do \emph{not} need to assume that $K_M$ is ample. We expect however that a similar phenomenon occurs for the Mabuchi functional, see Conjecture \ref{j-functionalwithoutampleness}.

We can strengthen the above results when $M$ is a surface by a more delicate analysis of the J-stability condition. 

\begin{theorem} If $M$ is a projective surface satisfying $\frac{4}{3}\gamma L_1 - L_2$ is nef, then $(M,L_1,L_2)$ is J-semistable provided $\gamma>0$. In particular if $L_2=K_M$, then $M$ is K-stable. \end{theorem}

This gives the most general currently known criterion proving K-stability of polarised surfaces with $\gamma>0$.

Finally, building on work of Sano and Seyyadali \cite{Sa,Sey2}, we prove that J-balanced metrics are the fixed points of a natural dynamical system. This gives a variational approach to the existence of such metrics and provides a natural interpretation of the energy functionals associated to the J-flow that appeared in \cite{SW,C1}.

\begin{theorem} Assume there exists a critical metric of the J-flow. Then for each $k\gg 0$, there is a natural map $$T_{k,\chi} :\Met(L_1^k) \rightarrow \Met(L_1^k)$$ defining a dynamical system which has the J-balanced metric as its (unique) fixed attractive point. \end{theorem}

 This also gives an algorithm to compute critical metrics, however we do not pursue this direction in the present work.

\subsection{The J-flow from symplectic geometry}
The J-flow is very natural from the point of view of symplectic geometry as we now briefly explain. Consider $M$ a compact symplectic manifold and $\omega,\chi$ are symplectic forms.  If one considers the manifold ${\rm Diff}(M)$  of diffeomorphisms $f:M\rightarrow M$ homotopic to
the identity and equipped with a natural symplectic form \begin{equation}\Upsilon_{\chi,\omega}(a,b)=\int_M \chi(a,b)\frac{\omega^n}{n!} \label{Omegachiomega}, \end{equation}
then there is a naturally associated moment map picture \cite{D9}. Here $a,b\in \Gamma(f^*(TM))$ as we identify the tangent space of ${\rm Diff}(M)$ at $f$ to the space of smooth sections of $f^*(TM)$. The group $\mathcal{G}$ of exact $\omega$-symplectomorphisms of $M$ acts on ${\rm Diff}(M)$
and preserves $\Upsilon_{\chi,\omega}$. Since we can identify the Lie algebra $Lie(\mathcal{G})$ with the set $$\{f\in C^{\infty}(M,\mathbb{R}), \,\, \int_M \,f\, {\omega^n}=0\},$$
we can express simply the associated moment map $\mu_J: {\rm Diff}(M)\rightarrow Lie(\mathcal{G})^*$ for the group action as
\begin{equation}\label{mu1}   \mu_J(f)=\frac{f^*(\chi)\wedge \omega^{n-1}}{\omega^n} - \gamma.
 \end{equation}
This moment map induces also a gradient flow $f_t$ of the function $\Vert \mu_J(f_t)\Vert^2$. Assuming that $M$ has an integrable almost-complex structure compatible with $\omega$ (in the sequel $M$ will always be K\"ahler), the $J$-flow is just the gradient flow expressed using $(f_t^*)^{-1}(\omega)$ on $M$, fixing the complex structure and varying the form within the K\"ahler class.
From \cite{C1}, it is known long time existence of Donaldson's J-flow for all time, that is \eqref{Jflow} admits a smooth solution for all $t\geq 0$. If it exists, Chen moreover shows that the critical metric is actually unique.

\subsection{CscK metrics and the J-flow}\label{s2}

While the Yau-Tian-Donaldson conjecture predicts the existence of a cscK metric is equivalent to K-stability, in practice it seems just as difficult to check K-stability in specific examples as it is to check the existence of a cscK metric. Thus, it is beneficial to give more explicit criteria that apply in concrete examples. For example, one can ask if given such a K\"ahler class that admits a cscK metric (for example a K\"ahler-Einstien metric, where many examples are known), if it is possible to describe nearby K\"ahler classes that have also cscK metrics. One way of achieving this goal is by understanding properties of the Mabuchi energy; recall an important conjecture of Mabuchi-Tian asserts that the existence of a cscK metric should be equivalent to the properness of the Mabuchi energy (a small modification is needed if $M$ admits holomorphic vector fields).

From the above discussion, it is therefore natural to search for K\"ahler classes with proper Mabuchi energy around a K\"ahler class that is endowed with a cscK metric. In this direction, it was first proven by Chen that if $[\chi]\in -c_1(M)<0$ and there exists a critical metric of the J-flow in the class $[\omega]$, then the Mabuchi energy for $[\omega]$ is bounded from below, see \cite{C2,We1}. This was strengthened by  Song and Weinkove who proved the properness of the Mabuchi energy \cite{SW}. Thus Mabuchi-Tian's conjecture predicts the existence of a cscK metric in such $[\omega]$.
  
 Donaldson observed that a necessary condition for the existence of a critical metric is the following inequality on the Chern classes \begin{equation}\label{Dcond} n\gamma[\omega]-[\chi]= n\gamma c_1(L_1)-c_1(L_2)>0. \end{equation} In dimension 2, it is proved in \cite{C1,We2} that the condition is sufficient,  the problem being proved equivalent to solve a Monge-Amp\`ere equation in \cite{C1} (see also the study of the convergence of the J-flow in \cite{We2}). However, the condition is \emph{not} sufficient in higher dimensions, see \cite{L-Sz}. Ideally, one could hope that an optimal Chern inequality similar to \eqref{Dcond} (or more generally an algebro-geometric condition) would imply the existence of critical metrics of the J-flow, and that this condition would be relatively simple to check, at least for general type manifolds, in view of \cite[Section 4]{Ross}. Keeping in mind the Yau-Tian-Donaldson conjecture, one can expect that the right algebro-geometric condition for the existence of a critical metric would provide the K-stability of the class.

The problem to find a good algebro-geometric condition equivalent to the existence of a critical metric is subtle. For instance, it is known from \cite{SW,We1} that
if there exists a K\"ahler metric $\omega'=\omega + \ddbar \psi\in c_1(L_1)$ satisfying 
\begin{equation}\label{cond1}(n\gamma \omega' - (n-1) \chi) \wedge (\omega')^{n-2} \wedge u \wedge \bar u >0\end{equation}
for all $(1,0)$-form $u$, then there is convergence of the flow towards a critical metric. Conversely if there is convergence such a metric $\omega'$ does exist. Unfortunately, it seems pretty hard to check \eqref{cond1} in practice. For instance it is not even clear whether it depends on the class only (and not on the forms) as pointed out in \cite{L-Sz}. 

The most general result using the J-flow to prove properness of the Mabuchi functional for certain K\"ahler classes is due to Li-Shi-Yao \cite{Li-Shi-Yao}. We remark that they are also able to use J-flow techniques to prove properness of the Mabuchi functional for certain K\"ahler classes also when $c_1(M)$ is \emph{positive}, i.e. $M$ is Fano (see \cite{Derv3} for a similar result proven by a direct analysis of the Mabuchi functional).

Very recently, Lejmi and Sz{\'e}kelyhidi \cite{L-Sz} have made an important contribution to the subject by proposing an algebro-geometric condition (that we call J-stability, see Section \ref{sectalg-geom}) using the technology of deformation to the normal cone and modelled on the K-stability theory. Their conjecture has been proven in the toric case \cite{Sz-Co}.

%\bigskip
%\todo[inline]{to rewrite below}
%In this paper we explain how the techniques of \cite{D1,Fi1,Ke2} can be modified to obtain a finite dimensional approach to Donaldson J-flow, providing a GIT construction for critical metrics and relating it to the work of Lejmi-Sz{\'e}kelyhidi via a generalization of the notion of Chow stability.

\subsection{Organisation of the paper}
In section \ref{sect11}, we introduce a notion of \textit{J-balanced} embeddings for $X$ in certain projective spaces depending on a quantum parameter $k\gg0$. This notion of J-balanced embedding is based on the symplectic formalism of moment maps and has an algebro-geometric interpretation that will become clear in Section \ref{sectalg-geom}. By pulling-back  the Fubini-Study metric, such embeddings provide algebraic metrics that we call \textit{J-balanced}  metrics.
We also introduce certain flows on the space of Bergman metrics of level $k$, called J-balancing flows, that converge towards J-balanced metrics if they do exist. More precisely, we prove that at the quantum limit $k\rightarrow +\infty$, the J-balancing flows converge towards to Donaldson's J-flow (Theorem \ref{thmJbal1}, Theorem \ref{mainthm1}).  
\par In Section \ref{varJflow}, building on the work of Sano and Seyyedali \cite{Sa,Sey2}, we prove that J-balanced metrics when they do exist, are actually attractive fixed points of a natural dynamical system (Theorem \ref{conviter}, Corollary \ref{Jcoroiter}). This gives a variational approach to the existence of such metrics and provides a natural interpretation of the energy functionals associated to the J-flow that appeared in \cite{SW,C1}. 
\par In section \ref{sectalg-geom}, we introduce a generalisation of the notion of Chow stability adapted to our context and prove that the existence of J-balanced metrics implies the stability in that sense (Theorem \ref{balancedimpliesstable}). In particular, our Chow stability condition is a necessary condition for the existence of a critical metric to the J-flow (Corollary \ref{corChow}). We describe how it is related to the notion of J-stability of Lejmi-Sz{\'e}kelyhidi. 
Building on the blowing-up formalism of Wang and Odaka for expressing weights, we are able to relate the notion of J-stability  and the classical notion of K-stability, giving a simple criterion for J-stability which corresponds to (weaker) algebraic version of a result of Song-Weinkove \cite{SW} (Theorem \ref{sufficientconditionjstability}, Theorem \ref{jstablesurfaces}). Eventually we relate J-semistability to K-stability  (Theorem \ref{jimpliesk}).\\

{\small
\noindent {\bf Acknowledgments.} {\small
 The first author would like to thank Julius Ross for helpful discussions on Hilbert and Chow stability. The second author is very grateful to Joel Fine for illuminating conversations on the subject of balanced embeddings throughout the years, and also to Ben Weinkove.\\
 The first author was funded by a studentship associated to an EPSRC Career Acceleration Fellowship (EP/J002062/1). \\
 The work of the second author has been carried out in the framework of the Labex Archim\`ede (ANR-11-LABX-0033) and of the A*MIDEX project (ANR-11-IDEX-0001-02), funded by the ``Investissements d'Avenir" French Government programme managed by the French National Research Agency (ANR). The second author was also partially supported by supported by the ANR project EMARKS, decision No ANR-14-CE25-0010. As By-fellow of Churchill college, he is very grateful to Churchill college and Cambridge University for providing him excellent conditions of work during his stay. \\
 Both authors are very grateful to the referees for their detailed reading of the manuscript and their suggestions for its improvement.
}}

\section{Finite dimensional approach to the J-flow \label{sect11}}
Given the Hermitian metric $h\in {\rm Met}(L^k_1)$ with positive curvature, one can consider the 
Hilbertian map \label{Hilbkchi}
$$Hilb_{\chi}=Hilb_{k,\chi} : {\rm Met}(L_1^k) \rightarrow {\rm Met}(H^0(L_1^k)) $$
such that $$Hilb_{\chi}(h)=\frac{1}{\gamma}\int_M \langle .,. \rangle_h \,\,\, {\chi \wedge {c_1(h)^{n-1}}}$$ is the $L^2$ metric induced by the fibrewise $h$  and the volume form $\chi\wedge c_1(h)^{n-1}$. 
On the other hand, one can consider the injective Fubini-Study maps $FS=FS_k :  {\rm Met}(H^0(L_1^k)) \rightarrow {\rm Met}(L_1^k), $
such that for $H\in {\rm Met}(H^0(L_1^ k) )$, $\{s_i\}$ an $H$-orthonormal basis of $H^0(X,{L_1}^k)$ and for all $p\in X$,
$$\sum_{i=1}^{\dim H^0(X,{L_1}^k) }\vert s_i(p)\vert^2_{FS_k(H)}= \frac{\dim H^0(X,{L_1}^k)}{\Vol_{{L}_1}(X)} ,$$
which means that we fix pointwisely the metric $FS_k(H)\in  {\rm Met}({L}^k)$. The curvature of $FS(H)$ is the pull-back of the Fubini-Study metric living in the projective space, using the embedding defined by the $H$-orthonormal basis $\{s_i\}$.

We also define the map \label{Tkchi}
$$T_{k,\chi}= FS\circ Hilb_{\chi}.$$

\begin{definition}[J-balanced metric]
A fixed point $h_k$ of the map $$T_{k,\chi}: {\rm Met}(L_1^k) \rightarrow {\rm Met}(L_1^k)$$ is called a J-balanced metric at level $k$.
\end{definition}

Let us denote in the sequel $N=N_k=\dim H^0(L_1^k)-1$. We introduce a moment map setting in finite dimensions. Let us consider first 
$\mu_{FS} : \mathbb{P}^{N} \rightarrow \sqrt{-1}Lie(U(N+1))$ which is a moment map for the $U(N+1)$ action and the Fubini-Study metric $\omega_{FS}$ on $\mathbb{P}^{N}$. Given homogeneous unitary coordinates, one sets explicitly  $\mu_{FS}=(\mu_{FS})_{\alpha,\beta}$ as
\begin{equation}
\left(\mu_{FS}([z_0,...,z_N])\right)_{\alpha,\beta}=\frac{z_\alpha \bar{z}_\beta}{\sum_i |z_i|^2}. \label{mu}
\end{equation}
Then,  given a holomorphic embedding $\iota:M\hookrightarrow  \mathbb{P}H^0(L_1^k)^*$, and the Fubini Study form $\omega_{FS}$ on the projective space, define
\begin{equation}\mu_{k,\chi}(\iota)=\frac{1}{\gamma}\int_M \mu_{FS}(\iota(p)) \;{\chi\wedge \iota^*(\omega_{FS}^{n-1})} (p).\label{Jmomentmap} \end{equation}
An alternative, but equivalent, point of view is to consider $\mu_{k,\chi}$ as a map from the space of holomorphic bases of $\mathbf{s}=\{s_i\}$ of $H^0(L_1^k)$ as
a basis determines a unique Hermitian inner product $H$ for which it is orthonormal. Then, this inner product induces a Fubini-Study metric $\omega_{FS,\mathbf{s}}$ and thus we can consider
$$\mu_{k,\chi}(\mathbf{s})=\frac{1}{\gamma}\int_M \mu_{FS}([s_1(p),...,s_N(p)]) \;{\chi\wedge \omega_{FS,\mathbf{s}}^{n-1}} (p).$$
Later we will be interested in special Hermitian metrics $H$ (associated to a particular basis) and so we shall write $\mu_{k,\chi}(H)$.

\begin{proposition}\label{claim}
 The map $\mu_{k,\chi}$ is a moment map for the $U(N+1)$ action over the space of all bases of $H^0(L_1^k)$ with respect to the symplectic structure defined by Equation \eqref{struct}. 
\end{proposition}
\begin{proof}
We follow essentially the techniques developped in \cite[Theorem 8.5.1]{Gaud}. Given $A\in Lie(GL(N+1))$, we denote $\hat{A}$ the induced action on ${\mathbb{P}^N}$ and 
$\hat{A}^\perp=\hat{A}-\hat{A}\vert_{TM}$ where  $\hat{A}\vert_{TM}$ stands for the orthogonal projection on $TM$ with respect to the Fubini-Study metric.
On the space of bases identified to $GL(N+1)$, we have a natural symplectic structure $\varpi$ defined at $\mathbf{s}$ by
\begin{align}
\varpi(\hat{A},\hat{B})=&\frac{1}{\gamma}\int_M \omega_{FS,\mathbf{s}}(\hat{A}^\perp,\hat{B}^\perp) \; {\chi \wedge \omega_{FS,\mathbf{s}}^{n-1}} \nonumber \\ 
 &+\frac{1}{n\gamma}\int_M g_{FS,\mathbf{s}}\left(\chi, i_{\hat{A}\vert_{TM}}\omega_{FS,\mathbf{s}}\wedge i_{\hat{B}\vert_{TM}}\omega_{FS,\mathbf{s}}\right)\; {\chi \wedge \omega_{FS,\mathbf{s}}^{n-1}},\label{struct}
\end{align}
where $A,B\in Lie(GL(N+1))$ and $g_{FS}=\omega_{FS,\mathbf{s}}(\cdot, J\cdot)$ is the associated metric to the K\"ahler form $\omega_{FS,\mathbf{s}}$. 
We shall use the following fact (see for instance the proof of \cite[Lemma 3.2.1]{Gaud}): for any 1-forms $\alpha,\beta$ and  K\"ahler forms $\omega,\chi$, one has the identity  
$$\frac{1}{n}g(\chi,\alpha\wedge \beta)\omega^n = \omega(\alpha,\beta)\wedge \chi\wedge \omega^{n-1}- (n-1)\alpha\wedge \beta \wedge \chi \wedge \omega^{n-2},$$
where $g$ is the associated metric to $\omega$. Hence we can compute 
\begin{align*}
\gamma\varpi(\hat{A},\hat{B})=&\int_M \omega_{FS,\mathbf{s}}(\hat{A}^\perp,\hat{B}^\perp) \; {\chi \wedge \omega_{FS,\mathbf{s}}^{n-1}}  \\
&+\int_M \omega_{FS,\mathbf{s}}\left(i_{\hat{A}\vert_{TM}}\omega_{FS,\mathbf{s}}, i_{\hat{B}\vert_{TM}}\omega_{FS,\mathbf{s}}\right) {\chi \wedge \omega_{FS,\mathbf{s}}^{n-1}} \\
&-(n-1) \int_M i_{\hat{A}\vert_{TM}}\omega_{FS,\mathbf{s}}\wedge i_{\hat{B}\vert_{TM}}\omega_{FS,\mathbf{s}}\wedge  {\chi \wedge \omega_{FS,\mathbf{s}}^{n-2}},\\
=& \int_M \omega_{FS,\mathbf{s}}(\hat{A},\hat{B}) \; {\chi \wedge \omega_{FS,\mathbf{s}}^{n-1}} \\
&-(n-1) \int_M i_{\hat{A}\vert_{TM}}\omega_{FS,\mathbf{s}}\wedge i_{\hat{B}\vert_{TM}}\omega_{FS,\mathbf{s}}\wedge  {\chi \wedge \omega_{FS,\mathbf{s}}^{n-2}},\\
=& \int_M \omega_{FS,\mathbf{s}}(\hat{A},\hat{B}) \; {\chi \wedge \omega_{FS,\mathbf{s}}^{n-1}} \\
&-(n-1) \int_M \partial \mu_{FS}(\hat{A}) \wedge \bar{\partial} \mu_{FS}(\hat{B}) \wedge  {\chi \wedge \omega_{FS,\mathbf{s}}^{n-2}},\\
=& \int_M \omega_{FS,\mathbf{s}}(\hat{A},\hat{B}) \; {\chi \wedge \omega_{FS,\mathbf{s}}^{n-1}} \\
&+(n-1) \int_M  \tr(\mu_{FS}A) \bar{\partial} \partial \mu_{FS}(\hat{B}) \wedge {\chi \wedge \omega_{FS,\mathbf{s}}^{n-2}},\\
=\hspace{-0.02cm}&\hspace{-0.02cm}  \big\langle d \int_M  \mu_{FS}\; \chi\wedge \omega_{FS,\mathbf{s}}^{n-1}(\hat{B}),A \big\rangle.
\end{align*}
During the computation we used the fact that the embedding given by $\mathbf{s}$ is holomorphic and $\mu_{FS}$ is a moment map on the projective space. 
Moreover $\mu_{k,\chi}$ is $Ad$-equivariant as the integral of the $Ad$-equivariant moment map $\mu_{FS}$. 
\end{proof}
%\medskip
% The couple $(M,L^k)$ provides a point in the Chow scheme $$\mathfrak{C}_{n,d,h^0(M,f}=\{ X\subset \mathbb{P}H^0(L^k)^*, \dim(X)=n, Vol_L(X)=d\}$$ One can consider a K\"ahler form $\Sigma$ on $\mathfrak{C}_{n,d,h^0(M,L^k)}$ over its smooth part, by considering the evaluation map $ev:\mathfrak{C}_{n,d,h^0(M,L^k)}\times \mathbb{P}H^0(L^k)^*$
Now $SU(N+1)$ acts isometrically on the spaces of holomorphic bases with the moment map given by 
$$\mathbf{s}\mapsto-\sqrt{-1}\left(\mu_{k,\chi}(\mathbf{s}) - \frac{\tr(\mu_{k,\chi}(\mathbf{s}))}{N+1}\Id_{N+1}\right)\in \sqrt{-1}Lie(SU(N+1)).$$
In the {Bergman space} of metrics $GL(N+1)/U(N+1),$ we have a preferred metric associated and this is precisely a J-balanced metric.

\begin{definition}[J-balanced embedding]\label{mu0kchi}
The embedding $\iota$ is J-balanced if and only if $$\mu^0_{k,\chi}(\iota):=\mu_{k,\chi}(\iota) - \frac{\tr(\mu_{k,\chi}(\iota))}{N+1}\Id_{N+1}=0.$$
\end{definition}

A J-balanced  embedding corresponds (up to $SU(N+1)$-isomorphisms) to a J-balanced metric $\iota^*\omega_{FS}$ by pull-back of the Fubini-Study metric from $\mathbb{P}H^0(L^k_1)^*$ so our two definitions of J-balanced metric and embedding actually agree. A notion of balanced basis can also be derived in an obvious way using previous discussion.
Note that for $H\in {\rm Met}(H^0(L_1^k))$,  it also makes sense to consider $\mu_{k,\chi}(h)$ where $h=FS(H)\in {\rm Met}(L_1^k)$, i.e when $h$ belongs to the space of  \textit{Bergman} type fibrewise metric that we identify with $\mathcal{B}=\mathcal{B}_k$.

\medskip 

\par On the other hand, seen as a Hermitian matrix, $\mu^0_{k,\chi}(\iota)$ induces a vector field on $\mathbb{P}^{N}$. We are lead to study the following flow
\begin{equation*}
\frac{d \iota(t)}{dt} =  -  \mu^0_{k,\chi}(\iota(t)),  
\end{equation*}
and we call this flow the J-balancing flow. To fix the starting point of this flow, we choose a K\"ahler metric
$\omega=\omega(0)$ and we construct a sequence of Hermitian metrics $h_k(0)$ such that $\omega_k(0):=c_1(h_k(0))$ converges
smoothly to $\omega(0)$ providing a sequence of embeddings $\iota_k(0)$ for $k\gg 0$. For technical reasons, we decide to rescale this flow by 
considering the following ODE,
\begin{equation}
\frac{d \iota_k(t)}{dt} =  - k^2 \mu^0_{k,\chi}(\iota_k(t)),  \label{resbalflowJ}
\end{equation}
which we call the \textit{rescaled J-balancing flow}. In the following Subsections \ref{sub1}, \ref{sub2}, \ref{sub3}, we will study the behaviour of the sequence of K\"ahler metrics 
$$\omega_{k}(t)=\frac{1}{k}\iota_k(t)^*(\omega_{FS})$$ as $k$ tends to infinity.

\subsection{The limit of the rescaled J-balancing flow \label{sub1}}
In this section, we assume that the sequence $\omega_{k}(t) $ is convergent and we want to relate its limit to Equation (\ref{Jflow}). 
\begin{theorem}\label{thmJbal1}
Suppose that for each $t\in \mathbb{R}_+$, the metric $\omega_k(t)$ induced by Equation (\ref{resbalflowJ}) converges in
smooth topology to a metric $\omega_t$ and that this convergence is $\C^1$ in $t\in \mathbb{R}_+$. Then
the limit $\omega_t$ is a solution to Donaldson's J-flow (\ref{Jflow}) starting at $\omega_0=\lim_{k\rightarrow \infty} \omega_k(0)$.
\end{theorem}
The proof is similar to \cite[Theorem 3]{Ke2}. The only difference is that we are dealing with orthonormal basis of holomorphic sections $\{s_i\}$ of $L_1^k$  with respect to $Hilb_{k,\chi}(h^k)$. But in that case, the asymptotic of the Bergman function stands as 
\begin{equation}\label{bergJflow}\sum_{i=1}^{N_k+1} \vert s_i\vert^2_{h^k}= k^n \frac{\gamma \omega^n}{\chi \wedge \omega^{n-1}}+O(k^{n-1}) \end{equation}
 where $\omega=c_1(h)$ thanks to the following proposition, see \cite{Ti1} (we also refer to  
\cite{Ze, Ca} and \cite{Bou} where the first term of the asymptotic expansion is identified).
\begin{proposition}[Catlin-Tian-Yau-Zelditch expansion]\label{tian-b}
Let $(M,L)$ be a projective polarized manifold. Let $h\in {\rm Met}(L)$ be a metric such that its curvature $c_1(h)=\omega >0$ is a K\"ahler form. Assume $\Omega$ to be a smooth volume form. % with continuous density.
Then we have the following asymptotic expansion for $k\rightarrow \infty$, 
\begin{equation*}
\sum_{i=1}^{N+1}\vert s_i\vert^2 _{h^k} = k^n \frac{\omega^n}{\Omega} +O(k^{n-1}), 
 \end{equation*}
where $\{s_i\}$ is an orthonormal basis of $H^0(L^k)$ with respect to the $L^2$ inner product $\int_M h^k(.,.) \Omega=Hilb_{\Omega}(h^k)$. Here by $O(k^{n-1})$, we mean that for $r\geq 0$
$$\Big\Vert \sum_{i=1}^{N+1}\vert s_i\vert^2 _{h^k} - k^n \frac{\omega^n}{\Omega}\Big\Vert_{\C^r}\leq c_rk^{n-1}$$
where $c_r$ remains bounded if $h$ varies in a compact set (in smooth topology) in the space of Hermitian metrics with positive curvature.
\end{proposition}
 
In particular, the potentials $\beta_k=-k\tr(\mu_{k,\chi}^0\mu_{FS})$ converge in smooth topology to the potential \begin{equation} 1 - \frac{\chi \wedge \omega^{n-1}}{\gamma \omega^n}\label{limit}\end{equation} when $k\rightarrow +\infty$. We also have
 \begin{proposition}\label{c-1-conv}
Let $h(t)\in {\rm Met}(L_1)$ be a path of Hermitian metrics on $L_1$ with $c_1(h(t))>0$. Let us consider $h_k(t)=FS(Hilb_{k,\chi}(h(t)^k))^{1/k}$ the path of induced Bergman metrics. Then $\frac{\partial h_k(t)}{\partial t}$ converges to $\frac{\partial h(t)}{\partial t}$ as $k\rightarrow +\infty$ in $\C^\infty$ topology. This convergence is uniform if $h(t)$ belongs
to a compact set in the space of positively curved Hermitian metrics on $L_1$.
\end{proposition}
 \begin{proof}
  This is obtained easily by a simple modification of the proof of \cite[Proposition 2.3]{Ke2} or \cite[End of section 1.4.1 and Theorem 9]{Fi1}.
 \end{proof}

\noindent \textit{Proof of Theorem \ref{thmJbal1}}. \,\,
Let's write $\omega_t=\omega+ \ddbar\phi_t$. By assumption, $\dot\phi_t$ is continuous and can be normalized to be unique by demanding that it has vanishing integral. Consider the potential  $\beta_k(t)$ induced by  the embedding $\iota_k(t)$ given by the rescaled J-balancing flow. The integral of $\beta_k(t)$ is zero. Therefore, with Proposition \ref{c-1-conv}, this sequence of potentials converges to $\dot\phi_t$. Hence, together with \eqref{limit}, the theorem is proved. \qed

\subsection{Convergence result for the rescaled J-balancing flow}\label{sub2}

\begin{theorem}\label{mainthm1}
Fix $T>0$. For any  $t\in [0,T]$, the sequence $\omega_{k}(t)$ converges in $\C^\infty$ topology to the solution of Donaldson's J-flow \eqref{Jflow}
with $\phi_0=0$ and $\omega= \lim_{k\rightarrow \infty}\omega_{k}(0)$. Furthermore, the convergence is $\C^1$ in the variable $t$. If there is a critical metric, then there is convergence for all $t>0$.
\end{theorem}

The last part of the theorem is a consequence of the long time existence of the flow and the fact that when there is a critical metric, the J-flow converges towards this critical metric \cite[Theorem 1.1, $(i)\Leftrightarrow (ii)$]{SW}. Thus the metrics involved in the J-flow belong to a compact set in the space of smooth K\"ahler metrics when there is a critical metric. The proof of Theorem \ref{mainthm1} will occupy Subsections \ref{firstorderJ}, \ref{highJ}, \ref{L2J}, and \ref{projestiJ}.

\subsubsection{The $Q_k$ operator}

In this section, we recall the following important technical result, see \cite[Theorem 1]{L-M}, \cite[Section 6]{M-M2}. Note that the $\C^r$ estimate below holds for any $f\in C^\infty(M,\mathbb{R})$.
\begin{theorem}\label{quantlaplacien}
Let us consider $h\in {\rm Met}(L)$ with positive curvature on an ample line bundle $L$, and $\omega=c_1(h)$ the induced K\"ahler form, $\Omega$ a smooth positive volume form and $\{s_a\}$ orthonormal basis of $H^0(L^k)$ with respect to $Hilb_\Omega(h^k)$.  Then the operator  on $C^\infty(M,\mathbb{R})$ given by
\begin{equation*}\label{sQk}
Q_k(f)(p)= \frac{1}{k^n}\int_M \sum_{a,b}{\langle s_a, s_b\rangle_{h^k} (q)\langle s_a, s_b\rangle_{h^k}(p)}f(q)\Omega(q),
\end{equation*}
approximates the operator $\frac{\omega^n}{\Omega}\mathrm{exp}(-\frac{\Delta}{4\pi k})$ in the following sense. For any $r\in \mathbb{N}^*$, there exists $C>0$ such that for all $k$ sufficiently large  and any function $f\in C^\infty(M,\mathbb{R})$, one has
\begin{align}
\Big\Vert \left(\frac{\Delta}{k}\right)^r \left(Q_k(f)-\frac{\omega^n}{\Omega}\mathrm{exp}\left(-\frac{\Delta}{4\pi k}\right)f\right)\Big\Vert_{L ^2}&\leq \frac{C}{k}\Vert f \Vert_{L^2} \label{ineq11} \\
\Vert Q_k(f)-\frac{\omega^n}{\Omega}f\Vert_{\C^r} &\leq \frac{C}{k} \Vert f \Vert_{\C^{r+2}} \label{ineq12}
\end{align}
where the norms are taken with respect to the induced K\"ahler form obtained from the fibrewise metric on the polarisation $L$ and $\Delta$ is the Laplace operator for the induced K\"ahler metric. The estimate is uniform when the metric varies in a compact set of smooth Hermitian metrics with positive curvature.
\end{theorem}

\subsubsection{First order approximation}\label{firstorderJ}
We know that from any starting point $\omega=\omega_0$, there exists a solution $$\omega_t=\omega+\ddbar \phi_t$$ to the J-flow for $t>0$. We can write $\omega_t=c_1(h_t)$ where $h_t$ is a sequence of Hermitian metrics on the line bundle $L_1$. 
Furthermore, we can construct a natural sequence of Bergman metrics \label{hath2} $$\hat{h}_k(t)=FS(Hilb_{\chi}(h_t^k))^{1/k}$$ by pulling back the Fubini-Study metric using sections which are orthonormal with respect to the inner product $$\frac{1}{k^n}\frac{1}{\gamma}\int_M h_t(.,.)^k\chi \wedge c_1(h)^{n-1}.$$
Using Proposition \ref{tian-b}, we obtain the asymptotic behavior $$\hat{h}_k(t)=\left(\frac{\gamma k^n c_1(h_t)^n}{\chi\wedge c_1(h_t)^{n-1}}+O\left(\frac{1}{k}\right)\right)^{1/k}h_t$$ for $k\gg 1$. Thus, the sequence  $\hat{h}_k(t)$  converges to $h_t$ as $k\to \infty$. 
\par On the other hand, the rescaled J-balancing flow provides a sequence of metrics $\omega_k(t)=c_1(h_k(t))$ which are solutions to (\ref{resbalflowJ}). Note that by construction, we fix $h_k(0)=\hat{h}_k(0)$ for the starting point of the rescaled J-balancing flow.
\medskip 
\par In this section, we wish to evaluate the distance between the two metrics $h_k(t)$ and $\hat{h}_k(t)$. Since we are dealing with algebraic metrics,  we have the (rescaled) metric on Hermitian matrices given by $d_k(H_0,H_1)=\left(\frac{\tr \;(H_0-H_1)^2}{k^2}\right)^{\frac{1}{2}}$ on ${\rm Met}(H^0(L^k_1))$ which induces a metric on ${\rm Met}(L_1)$,  that we denote by $\mathrm{dist}_k$. 

\begin{proposition}
One has for $t\in [0,T], $
$$\mathrm{dist}_k(h_k(t),\hat{h}_k(t))\leq \frac{C}{k},$$ for some constant $C>0$ independent of $k$ and $t$. 
\end{proposition}
\begin{proof}
Let us consider $e^{\phi(t)}h_0$ a family of Hermitian metrics with positive curvature,  and denote $$\omega_t=c_1(e^{\phi(t)}h_0).$$ The infinitesimal change at $t$ in the $L^2$ inner product induced by this path and the induced volume form is given by 
$$\hat{U}_{\alpha,\beta}(t)=\frac{1}{\gamma k^n}\int_M \langle s_\alpha, s_\beta\rangle\, \left(  \left( k\dot{\phi}(t) +\Delta_{\omega_t}\dot{\phi}(t)\right)  \chi\wedge\omega_t^{n-1} - \tilde{\Delta}_{\omega_t}\dot{\phi}(t)\omega_t^n  \right)$$
where $\Delta_{\omega_t}$ is the Laplacian with respect to $\omega_t$ and $\tilde{\Delta}_{\omega_t}$ is given by the Laplacian-type operator
$$\tilde{\Delta}_{\omega_t}u=\frac{1}{n}\omega_t^{k\bar j}\omega_t^{i\bar l} \chi_{i\bar j} \partial_k \partial_{\bar l}u.$$
Here $\{s_\alpha\}$ is  an orthonormal basis of $H^0(L_1^k)$ with respect to the $L^2$-inner product 
$$\frac{1}{\gamma k^n}\int_M e^{k\phi(t)} \chi \wedge \omega_t^{n-1}.$$ 
The formula is obtained by noticing that the variation occurs with respect to the fibrewise metric and the induced volume form. 
 Now, if furthermore $\phi(t)$ is a solution to the J-flow, this infinitesimal change  is given at $\hat{h}_k(t)$ as
\begin{equation*}
\hat{U}_{\alpha,\beta}(t)=\frac{1}{\gamma k^n}\int_M \langle s_\alpha, s_\beta\rangle \left(k\left(1-\frac{\chi \wedge \omega_t^{n-1}}{\gamma\omega_t^n} \right)+O(1)\right)\chi \wedge \omega_t^{n-1}
\end{equation*}
with $\{s_\alpha\}$ satisfy the same assumption as above. 
\par  On another hand, the tangent (at the same point $\hat{h}_k(t)$) to the rescaled J-balancing flow (\ref{resbalflowJ}) is given directly by the moment map $\mu^0_{k,\chi}$, and we write the infinitesimal change of the $L^2$ metric as 
\begin{equation*}
 {U}_{\alpha, \beta}(t) = \frac{k^2}{\gamma k^n}\int_M \left(\frac{\delta_{\alpha\beta}}{N+1}- \frac{\langle s_\alpha, s_\beta \rangle}{\sum_{i=1}^{N+1} \vert s_i\vert^2 }\right) \chi\wedge \omega_{FS}^{n-1},
 \end{equation*}
where $s_i$ are $L^2$ orthonormal with respect to the $L^2$ inner product induced by $h(t)^k$ and $\chi\wedge \omega_t^{n-1}$. Again, using the fact that $\omega_{FS}=k\omega_t+O(k^{-1})$ and from Proposition \ref{tian-b}, one has asymptotically $${U}_{\alpha, \beta}(t)  =\hat{U}_{\alpha,\beta}(t) + \frac{1}{k^n}\int_M \langle s_\alpha, s_\beta \rangle O(1) \;\chi\wedge \omega_t^ {n-1}.$$ 
Here the term $O(1)$ stands implicitly for a (smooth) function which is bounded independently of the variables $t$ and $k$. 
Thus, one has $$\frac{\tr \;(\hat{U}_{\alpha, \beta}(t) - U_{\alpha,\beta}(t)) ^2}{k^2} = \big\langle \frac{1}{k}O(1) , Q_k\left( \frac{1}{k}O(1)\right) \big\rangle_{L^2}.$$
We can use Theorem \ref{quantlaplacien}, Inequality (\ref{ineq11}) to obtain that $$\frac{\tr \;(\hat{U}_{\alpha, \beta}(t) - U_{\alpha,\beta}(t)) ^2}{k^2} =O(k^{-2}).$$ This shows that $d_k(\hat{U}_{\alpha, \beta}(t), U_{\alpha,\beta}(t)) )=O(1/k).$ 
If we denote by \label{tildehk2}$\tilde{h}_k(t)$ the rescaled J-balancing flow passing through $\hat{h}_k(t_0)$ at $t=t_0$, we have just proved that 
$\tilde{h}_k(t)$ and $\hat{h}_k(t)$ are tangent up to an error term in $O(1/k)$ at $t=t_0$. On the other hand, it is clear that
 $\tilde{h}_k(t)$ and $h_k(t)$ are close when $t\rightarrow \infty$, because they are  obtained through the gradient flow of the same moment map and this gradient flow is distance decreasing (see also \cite[Theorem 1]{C1}). Thus $\mathrm{dist} (\tilde{h}_k(t),h_k(t))=O(1/k)$. This finally proves the result.
\end{proof}

\subsubsection{Higher order approximation}\label{highJ}

In this section, we only describe the main differences with \cite[Section 4.2]{Ke2}. The key operator appearing in the linearisation of the problem is actually,
$$\eta \rightarrow \mathfrak{L}_t(\eta)=\frac{\partial \eta}{\partial t} -{\tilde{\Delta}}\eta.$$
We mean that it is sufficient to solve inductively equations of the form $\mathfrak{L}_t(\eta_i)= \gamma_{i,0}(\eta_1,..,\eta_{i-1})$ where $\gamma_{i,0}$ is smooth.
%Note that there is a lower bound for $\Lambda_{\omega_t} \chi$ along the flow, \cite[Equation (2.3)]{W2}. 
By the standard parabolic theory, a smooth solution $\eta_t$ of $$\{\mathfrak{L}_t(\eta)=\xi, \eta(0)=0, \xi\in C^{\infty}(M,\mathbb{R})\}$$ exists for all time 
$t\geq 0$. Using this remark, it is easy to modify the arguments of \cite[Theorem 4]{Ke2} in order to obtain the following result.

\begin{theorem} \label{highapprox}
Fix $T>0$. Given  solution $\phi_t$ for $t\in [0,T]$ to Donaldson's J-flow (\ref{Jflow}) and $k\gg 0$, there exist  functions $\eta_1,...,\eta_m$, $m\geq 1$, such that the deformation of $\phi_t$ given by the potential
$$\psi(k,t)=\phi_t+\sum_{j=1}^m \frac{1}{k^j}\eta_j(t)$$
satisfies
$$\mathrm{dist}_k(h_k(t),\overline{h}_k(t))\leq \frac{C}{k^{m+1}}$$
and 
$$\mathrm{dist}_k\left(\frac{\partial h_k(t)}{\partial t},\frac{\partial \overline{h}_k(t)}{\partial t}\right)\leq \frac{C}{k^{m}}.$$
Here \label{overhk2}$\overline{h}_k(t)=FS(Hilb_{\chi}(h_0^ke^{k\psi(k,t)}))^{1/k}\in {\rm Met}(L_1)$ is the induced Bergman metric from the potential $\psi$, $h_k(t)\in {\rm Met}(L_1)$ is the metric obtained by the rescaled J-balancing flow (\ref{resbalflowJ}), and $C$ is a positive constant independent of $k$ and $t$. 
\end{theorem}

\subsubsection{$L^2$ estimates in finite dimension}\label{L2J}
%\texorpdfstring{$L^2${L2}
 
In this section we provide analogues of Proposition 4.3, Lemma 4.1 and and Corollary 4.1 of \cite{Ke2}, where $\mu_{\Omega}$ is replaced by the moment map $\mu_{k,\chi}$. This is because all the proof of these results depend only on the integrand of the expression for the moment map $\mu_{k,\chi}$ given in \eqref{Jmomentmap}.

Fix $$H_A = \sum_{i,j} A_{ij} (s_i,s_j)=\tr(A\mu_{FS})\in C^{\infty}(M,\mathbb{R}),$$
where $A=(A_{ij})$ is a Hermitian matrix, $\{s_i\}$ is a basis of $H^0(L_1^k)$, and $(.,.)$ denotes the fibrewise Fubini-Study inner-product induced 
by the basis $\{s_i\}$.\\

We start this section by recalling the notion of $R$-boundedness in $C^r$ topology (see \cite[Secion 3.2]{D1}). The purpose of this definition is to avoid constants depending on $k$ in the forthcoming estimates. Let us fix a reference metric $\omega_{ref}\in c_1(L_1)$. We denote $\tilde{\omega}_{ref}=k\omega_{ref}$ the induced metric in $ kc_1(L_1)$.  We say that another metric $\tilde{\omega}\in kc_1(L_1)$ has
$R$-bounded geometry in $C^r$ if $$\tilde{\omega}>\frac{1}{R}\tilde{\omega}_{ref} \text{ and } \Vert \tilde{\omega}-\tilde{\omega}_{ref} \Vert_{C^r(\tilde{\omega}_{ref})}<R.$$ Moreover, we say that a basis $\{s_i\}$ of $H^0(L_1^k)$ is $R$-bounded if the Fubini-Study metric induced by 
the embedding of $M$ in $\mathbb{P}H^0(L_1^k)^*$ associated to $\{s_i\}$ has $R$-bounded geometry. 

\begin{proposition} There exists $C>0$ independent of $k$, such that for any basis $\{s_i\}$
 of $H^0(L_1^k)$ with $R$-bounded geometry in $\C^r$ and any Hermitian matrix $A$, 
$$\Vert H_A \Vert_{\C^r} \leq C \Vert \mu_{k,\chi}(\iota) \Vert_{op}\Vert A \Vert $$ 
where $\iota$ is the embedding induced by $\{s_i\}$. 
\end{proposition}

In the above proposition, the constant $C$ depends actually on the parameters $(M,L_1,r,\chi,\omega_{ref})$. 

\begin{lemma}\label{l5}
Let us fix $r\geq 2$. Assume that for all $t\in [0,T]$, the family of basis $\{s_i\}(t)$ of $H^0(L_1^k)$ have $R$-bounded geometry. Let us define by $h(t)$ the family of Bergman metrics induced by $\{s_i\}(t)$. Then
the induced family of Fubini-Study metrics $\tilde{\omega}(t)$ satisfy 
$$\Vert \tilde{\omega}(0)-\tilde{\omega}(T)\Vert_{\C^{r-2}}< C \sup_{t}\Vert \mu_{k,\chi}(\iota(t))\Vert_{op}\int_0^T \mathrm{dist}(h(s),h(0))ds,$$
and also
\begin{align*}
 \Big\Vert \frac{\partial\tilde{\omega}}{\partial t}(0)-\frac{\partial \tilde{\omega}}{\partial t} (T)\Big\Vert_{\C^{r-2}}\hspace{-0.12cm} <& C^* \sup_t \Vert \mu_{k,\chi}(\iota(t))\Vert_{op}\int_0^T \mathrm{dist}(\frac{\partial h}{\partial s}(s),\frac{\partial h}{\partial s}(0)) ds \\
&\hspace{-0.04cm}+ C^* \sup_t  \Vert d\mu_{k,\chi}(\iota(t))\Vert_{op}\int_0^T \mathrm{dist}(h(s),h(0)) ds,
\end{align*}
where $C,C^*$ are uniform constants in $k$.
\end{lemma}

\begin{corollary}\label{R-bounded}
Let $\tilde{\omega}_k$ be a sequence of metrics with $R/2$-bounded geometry in $\C^{r+2}$ such that the norms $\Vert\mu_{k,\chi}(\tilde{\omega}_k)\Vert_{op}$ are uniformly bounded. Then, there is a constant $C>0$ independent of $k$ such that if
 $\tilde{\omega}$ has $\mathrm{dist}_k(\tilde{\omega}, \tilde{\omega}_k) <  C$, then $\tilde{\omega}$ has $R$-bounded geometry in $\C^r$. 
\end{corollary}

\subsubsection{Projective estimates\label{projestiJ}}

We collect here some projective estimates, following the lines of \cite[Section 5]{Fi1}.
\begin{proposition}\label{cor-tian}
 Let $h$ be a Hermitian metric on $L_1$ with curvature $\omega=c_1(h)>0$. Consider the  
sequence $h_k=FS(Hilb(h))\in {\rm Met}(L_1^k)$ of Bergman metrics, approximating $h$ after renormalisation, thanks to Proposition \ref{tian-b}. Let us call $$\mathfrak{I}_{k,\chi}=\frac{1}{\gamma}\int_M \langle s_i,s_j\rangle_{h^k} \chi\wedge \omega^{n-1},$$ where $\{s_i\}$ is a basis of holomorphic sections of $H^0(L_1^k)$ orthonormal with respect to $Hilb(h)$. Then, as $k\rightarrow + \infty$, $$\Vert \mu_{k,\chi}(h_k)-\mathfrak{I}_{k,\chi}\Vert_{op}\rightarrow 0$$
 and the convergence is uniform for $\omega$ lying in a compact subset of K\"ahler metrics in $ c_1(L)$.
\end{proposition}

\begin{proof}
%Note that $\mathfrak{I}_{k,\chi}$ does not depend on the choice of the {\it orthonormal basis} of holomorphic sections $\{s_i\}$. 
Because of the asymptotic expansion (Proposition \ref{tian-b}), we have $$\mu_{k,\chi}(h_k)=\frac{1}{\gamma}\int_M \langle s_i,s_j\rangle_{h^k} (1+O(1/k)){\chi\wedge \omega_{FS}^{n-1}}.$$
Then, we apply \cite[Lemma 28]{D1} that gives that for the operator nom,
$$\Big\Vert \frac{1}{\gamma}\int_M \langle s_i,s_j\rangle_{h^k} (1+O(1/k)){\chi\wedge \omega_{FS}^{n-1}}\Big\Vert_{op}\leq \Big\vert \frac{\chi\wedge \omega_{FS}^{n-1}}{\gamma\omega_{FS}^n}O(1/k)\Big\Vert_{L^\infty}.$$
The uniformity of the convergence is given by the uniformity of the expansion in the asymptotics.
\end{proof}

In the sequel we fix a point $b\in \mathcal{B}$. %A tangent vector $A\in T_b\mathcal{B}$ can be seen as an element of $Lie(U(N+1))$.
\begin{lemma}\label{Jl1}
 For any pair of Hermitian matrices $A,B \in T_b \mathcal{B}$, denote $\hat{A},\hat{B}$ the induced vector field on $\mathbb{P}^N$. One has
$$\tr(B d \mu_{k,\chi}(A))= \frac{1}{\gamma}\int_M (\hat{A}, \hat{B}) \chi \wedge \omega_{FS}^{n-1}-  \partial H_B \wedge\bar \partial H_A\wedge \chi \wedge \omega_{FS}^{n-1},$$
where $(.,.)$ denotes the Fubini-Study form induced on the tangent vectors.
\end{lemma}
\begin{proof}
This is contained in the proof of Proposition \ref{claim} but for the sake of clearness let us provide the proof. We have, using the fact that $\mu_{FS}$ is a moment map, 
\begin{align*}
          \tr(B d \mu_{k,\chi}(A))\hspace{-0.02cm}=& \frac{1}{\gamma}\int_M \tr(Bd\mu_{FS}(A))\chi \wedge \omega_{FS}^{n-1}\\
& + \frac{1}{\gamma}\int_M \tr(B\mu_{FS})L_{\hat{A}}(\chi \wedge \omega_{FS}^{n-1}) \\
                =& \frac{1}{\gamma}\int_M (\hat{A},\hat{B})\chi \wedge \omega_{FS}^{n-1} \\
 &-\frac{1}{\gamma}\int_M \tr(B\mu_{FS})\partial\bar\partial \tr(A\mu_{FS}) \wedge  \chi \wedge \omega_{FS}^{n-1} \\
=& \frac{1}{\gamma}\int_M\hspace{-0.02cm} (\hat{A},\hat{B}) \chi \wedge \omega_{FS}^{n-1}\hspace{-0.07cm} \\
 &\hspace{1cm}- \hspace{-0.07cm} \partial \tr(B\mu_{FS})\hspace{-0.02cm}\wedge\hspace{-0.02cm}\bar \partial \tr(A\mu_{FS})\hspace{-0.02cm}\wedge \hspace{-0.02cm}\chi \hspace{-0.02cm}\wedge \hspace{-0.02cm} \omega_{FS}^{n-1}.
         \end{align*}

\end{proof}

By the fact that $\mu_{FS}$ is a moment map, we have the following simple lemma.

\begin{lemma}\label{lFi1}
 Let $A,B\in  T_b \mathcal{B}$. Pointwise over $({\mathbb{P}^N})^*$, one has
$$H_AH_B + (\hat{A},\hat{B})=\tr(AB\mu_{FS}).$$
\end{lemma}

The next lemma is deduced from Lemmas \ref{Jl1} and \ref{lFi1}.
\begin{lemma}\label{l2}
 For any Hermitian matrices $A,B \in T_b \mathcal{B}$,
$$\tr(Bd\mu_{k,\chi}(A)) + \langle H_A,H_B\rangle_{L^2_1(M,\frac{1}{\gamma}\chi \wedge \omega_{FS}^{n-1})}=\tr(AB\mu_{k,\chi}),$$
where the ${L^2_1(M,\frac{1}{\gamma}\chi \wedge \omega_{FS}^{n-1})}$-norm is computed with respect to the volume form $\frac{1}{\gamma}\chi \wedge \omega_{FS}^{n-1}$ and the gradient induced by $\chi$.
\end{lemma}

We can derive from the two previous results the following corollaries.
\begin{lemma}
 For any Hermitian matrix $A\in T_b \mathcal{B}$,
$$\Vert H_A \Vert^2_{L^2_1(M,\frac{1}{\gamma}\chi \wedge \omega_{FS}^{n-1})}\leq \Vert A \Vert^2 \Vert \mu_{k,\chi}\Vert_{op}.$$ 
\end{lemma}
\begin{proof}
 From the last lemma, 
$$\Vert H_A \Vert^2_{L^2_1(M,\frac{1}{\gamma}\chi \wedge \omega_{FS}^{n-1})} = \tr(A^2\mu_{k,\chi})-\tr(Ad\mu_{k,\chi}(A)).$$ Now, by Lemma \ref{Jl1},  
$\tr(Ad\mu_{k,\chi}(A))\geq 0.$ The conclusion follows from the fact that $\tr(A^2\mu_{k,\chi})\leq  \Vert A \Vert^2 \Vert \mu_{k,\chi}\Vert_{op}$.
\end{proof}

\begin{lemma}\label{l3}
 For any Hermitian matrix $A\in T_b \mathcal{B}$, 
$$\Vert d\mu_{k,\chi}(A) \Vert_{op} \leq \Vert d\mu_{k,\chi}(A) \Vert \leq 2 \Vert A \Vert \Vert \mu_{k,\chi} \Vert_{op}.$$
\end{lemma}
\begin{proof} From Lemma \ref{l2}, one has
\begin{align*}
\Vert d\mu_{k,\chi}(A) \Vert^2 &= \tr(d\mu_{k,\chi}(A)^2) \\
&= \tr(Ad\mu_{k,\chi}(A)\mu_{k,\chi})-\langle H_A,H_{d\mu_{k,\chi}(A)}\rangle_{L^2_1(M,\frac{1}{\gamma}\chi \wedge \omega_{FS}^{n-1})}\\
&\leq \Vert A \Vert 
\Vert d\mu_{k,\chi}(A)\Vert \mu_{k,\chi} \Vert_{op}-\langle H_A,H_{d\mu_{k,\chi}(A)}\rangle_{L^2_1(M,\frac{1}{\gamma}\chi \wedge \omega_{FS}^{n-1})}.
\end{align*} By Cauchy-Schwarz, \begin{align*}\vert \langle H_A,H_{d\mu_{k,\chi}(A)}\rangle_{L^2_1(M,\frac{1}{\gamma}\chi  \wedge \omega_{FS}^{n-1})}\vert  \\
&\hspace{-1cm}\leq \Vert H_A\Vert_{{L^2_1(M,\frac{1}{\gamma}\chi \wedge \omega_{FS}^{n-1})}}  \Vert H_{d\mu_{k,\chi}(A)}\Vert_{{L^2_1(M,\frac{1}{\gamma}\chi \wedge \omega_{FS}^{n-1})}},\end{align*}
and the previous lemma, we conclude the proof.
\end{proof}

Finally, we obtain as a consequence of our work the following proposition.
\begin{proposition} \label{p1}
 Let $b_0, b_1 \in \mathcal{B}$. Then,
$$\Vert \mu_{k,\chi}(b_1)\Vert_{op} \leq e^{2\mathrm{dist}_k(b_0,b_1)}\Vert \mu_{k,\chi}(b_0)\Vert_{op}.$$
\end{proposition}
\begin{proof}
The proof is similar to the proof of \cite[Proposition 4.5]{Ke2}, using Lemma \ref{l3}.
\end{proof}

We have now all the ingredients to proceed to the proof of the main result of this section.\\

\noindent {\textit{Proof of Theorem \ref{mainthm1}.}} \, The only difference with the proof of \cite[Theorem 1, p.26]{Ke2} is that here we need to estimate $\Vert \mathfrak{I}_{k,\chi}\Vert_{op}$ which is bounded from above by $\sup_M \frac{\chi \wedge \omega^{n-1}}{\gamma \omega^n}$ using \cite[Lemma 28]{D1}. This latter term is also bounded along the J-flow by a maximum principle argument. The other main ingredient of the proof is the uniformity in the evolving metrics, which is ensured by the fact that we are working in finite time or that we have smooth convergence.

For the sake of the clarity, we now provide a complete proof. 
Using Theorem \ref{highapprox}, for any $m>0$, we have obtained a sequence of K\"ahler metrics $$\omega(k;t)=c_1(h_0 e^{\psi(k,t)})$$ such that $\omega(k;t)$ converges, when $k\rightarrow +\infty$ and in smooth sense, towards the solution $\omega_t=c_1(h_0 e^{\phi_t})$ to the J-flow. 

Moreover, one has, for $k$ large enough and with
$\overline{h}_k(t)\in \mathcal{B}$ the Bergman metric associated to $h_0 e^{\psi(k,t)}\in {\rm Met}(L)$, the estimate 
\begin{equation}\label{estimate}
 \mathrm{dist}_k(h_k(t),\overline{h}_k(t))\leq \frac{C}{k^{m+1}},
\end{equation}
 where $h_k(t)$ is the metric induced by the rescaled J-balancing flow.
Consequently, in order to get the $\C^0$ convergence in $t$, all what we need to show is that 
\begin{equation}
\Vert \omega_k(t)-c_1(\overline{h}_k(t))\Vert_{\C^r(\omega_t)} \rightarrow 0. \label{aim} 
\end{equation}
The idea is to consider the geodesic in the Bergman space between these two points. 
\medskip

Firstly, we will get that along the geodesic from $\overline{h}_k(t)$ to $h_k(t)$ in $\mathcal{B}$,  $\Vert \mu_{k,\chi}\Vert_{op}$ is controlled uniformly if we can apply Proposition \ref{p1}. This requires to prove that $\overline{h}_k(t)$ is at a uniformly bounded distance of $h_k(t)$ and that $\Vert \mu_{k,\chi}(\overline{h}_k(t))\Vert_{op}$ is bounded in $k$. But, this comes from the fact that one can choose precisely $m\geq n+1$ in Inequality (\ref{estimate}) and one can apply Proposition \ref{cor-tian}.
\par Secondly, we show that the points along this geodesic have $R$-bounded geometry.  This is a consequence of Corollary \ref{R-bounded}, applied with the reference metric $\omega_t$ to the sequence $c_1(\overline{h}_k(t))$. On one side, $\Vert \mu_{\Omega}(\overline{h}_k(t))\Vert_{op}$ is under control as we have just seen. On another side, $c_1(\overline{h}_k(t))$ are convergent to $\omega_t$ in $\C^{\infty}$ topology (hence in $\C^{r+4}$ topology), thus they  have $R/2$-bounded geometry. Given $m\geq n+2$, one obtains, thanks to  Corollary \ref{R-bounded} and inequality (\ref{estimate}), that all the metrics along the geodesic from $\overline{h}_k(t)$ to $h_k(t)$ have $R$-bounded geometry in $\C^{r+2}$. 
\par Thirdly, we are exactly under the conditions of Lemma \ref{l5}. It follows that, by renormalising the metrics in the K\"ahler class $ c_1(L)$ and by (\ref{estimate}), that
\begin{align*}
\Vert k\omega_k(t)-kc_1(\overline{h}_k(t))\Vert_{\C^r(k\omega_t)}&\leq C \Vert \mu_{k,\chi}(\overline{h}_k(t))\Vert_{op}k^{n+2}\mathrm{dist}_k(h_k(t),\overline{h}_k(t)),\\
\Vert \omega_k(t)-c_1(\overline{h}_k(t))\Vert_{\C^r(\omega_t)}& \leq  C\Vert \mu_{k,\chi}(\overline{h}_k(t))\Vert_{op}k^{n+2-m-1+r/2},
\end{align*}
where we have used that the geodesic path from $0$ to $1$ is just a line. Here $C>0$ is a constant that does not depend on $k$. If we choose $m>r/2+1+n$, we get the expected convergence in $\C^r$ topology, i.e Inequality (\ref{aim}). Of course, this reasoning works to get the uniform $\C^0$ convergence in $t$ for $t\in \mathbb{R}_+$, because all the K\"ahler metrics $\omega_t$ that we are using are uniformly equivalent under our assumptions, and because we have uniformity of the expansion in Proposition \ref{tian-b} and Theorem \ref{quantlaplacien}.
\par We now prove that one has $\C^1$ convergence in $t$ of the flows $\omega_k(t)$. Again, we need to show the $\C^1$ convergence of $\omega_k(t)$ to $c_1(\overline{h}_k(t))$, because we already know the convergence of $c_1(\overline{h}_k(t))$ to $\omega_t$ by Proposition \ref{c-1-conv}.  We are under the conditions
 of Lemma \ref{l5} by what we have just proved above. So we have, using again that our path is a geodesic, 
\begin{align*}
\Big\Vert k\frac{\partial\omega_k(t)}{\partial t}-k\frac {\partial{c_1(\overline{h}_k(t))}} {\partial t}& \Big\Vert_{\C^r}\\
&\hspace{-1cm}\leq C^* \Vert \mu_{k,\chi}(\overline{h}_k(t))\Vert_{op}k^{n+2}\mathrm{dist}_k\left(\hspace{-0.07cm}\frac{\partial h_k(t))}{\partial t},\frac{\partial \overline{h}_k(t)}{\partial t}\hspace{-0.07cm}\right) \\
&\hspace{-0.3cm}+ C^*\Vert d\mu_{k,\chi}(\overline{h}_k(t))\Vert_{op}k^{n+2}\mathrm{dist}_k(h_k(t),\overline{h}_k(t)).
 \end{align*}
Here the $\C^r$ norm is computed with respect to $k\omega_t$. If we apply Lemma \ref{l3}, Theorem \ref{highapprox} (2nd inequality), we can bound from above the RHS of the last inequality, and get
\begin{align*}
\Big\Vert \frac{\partial\omega_k(t)}{\partial t}-\frac {\partial{c_1(\overline{h}_k(t))}} {\partial t}\Big\Vert_{\C^r(\omega_t)}\hspace{-0.1cm}\leq & C'\Vert \mu_{k,\chi}(\overline{h}_k(t))\Vert_{op}k^{n+2-m-r/2}\\
&+  C'' \Vert \mu_{k,\chi}(\overline{h}_k(t))\Vert_{op}k^{n+2+r/2}k^{-m-1}k^{-m-1}\\
\leq & C'''k^{n+2-m-r/2}.
\end{align*}
Finally, we choose $m>r/2+n+2$ to obtain $\C^1$ convergence. This completes the proof of Theorem \ref{mainthm1}.
\qed

\subsection{Convergence result for J-balanced metrics \label{sub3}}
The previous results are uniform if one assumes that the J-flow is convergent (and so $T$ can be chosen $T=+\infty$). Thus  a direct corollary of Theorem \ref{mainthm1}, the long time existence and convergence of the J-flow is the following. 

\begin{corollary}\label{cor1}
Consider $(M,L_1,L_2)$ a polarised manifold by $L_1,L_2$ such that that there exists a critical metric
solution of \eqref{critical}. Then for $k$ sufficiently large, there exists a sequence of J-balanced metrics on 
${\rm Met}(L_1^k)$ obtained as the limit of the rescaled J-balancing flow at time $t=+\infty$. Furthermore, the sequence of J-balanced metrics converges in smooth topology towards the critical metric when $k\rightarrow +\infty$. 
\end{corollary}

This is an analogue of the main result of \cite{D1}. Of course a more direct proof inspired from \cite{D1} could be used to derive Corollary \ref{cor1}. This would involve to the operator obtained from linearising the Bergman function close to the critical point $\omega_\infty$ i.e explicitly
\begin{equation*}\phi\mapsto {\tilde{\Delta}_{\omega_\infty}\phi}   %{(\Lambda_{\omega_\infty}\chi)^2}.
 \end{equation*}
%The denominator of the previous expression is a smooth positive function and thus this operator 
This operator is a uniformly elliptic second order operator. Its kernel consists of constant functions.
Eventually, the existence of J-balanced metrics can be seen as a necessary condition for the existence of critical metrics. As a consequence of our work, we recover the uniqueness of the critical metrics.

\section{Variational approach to the rescaled J-balancing flow\label{varJflow}}

\subsection{Convexity along geodesics}\label{J-func}

Let us consider the functional $$J_\chi:{\rm Met}(L_1)\rightarrow \mathbb{R}$$ on the space of smooth Hermitian metrics with positive curvature on $L_1$ defined up to an additive function by
\begin{equation}\label{Jdef}\frac{dJ_\chi(h_t)}{dt}=\frac{1}{\gamma}\int_M \dot{\phi_t} \; \chi \wedge c_1(h_t)^{n-1},\end{equation}
where $h_t=e^{-\phi_t}h_0$ is a smooth path in ${\rm Met}(L_1)$. Setting $$\omega_t=c_1(h_t)=\omega_0+\ddbar \phi_t,$$ a direct computation gives
\begin{equation*}\frac{d^2 J_\chi (h_t)}{dt^2}=\frac{1}{\gamma}\int_M \ddot{\phi_t}\;\chi\wedge \omega_t^{n-1}
- \dot{\phi_t}\tilde{\Delta}_{\omega_t}\dot{\phi_t}\;\omega_t^n + \dot{\phi_t}\Delta_{\omega_t}\dot{\phi_t}\;\chi\wedge \omega_t^{n-1}.
\end{equation*}

We shall use the same notation  $J_\chi$ as above for the induced functional defined  on ${\rm Met}(L_1^k)$ for $k>0$. 
Consider now the following functional \label{Imukchi0}$I_{\mu_{k,\chi}^0}:\mathcal{B}_k\rightarrow \mathbb{R}$ on the Bergman space defined by
$$I_{\mu_{k,\chi}^0}(H)=J_\chi \circ FS(H) +\frac{Vol_{L_1}(M)}{N+1}\log \det(H),$$
where $H\in \mathcal{B}_k$.
It is clear that the derivative of $J_\chi \circ FS$ at a point $H\in \mathcal{B}_k$ is given by
$$\frac{1}{\gamma}\sum_{i,j} \int_M (\delta H)_{i,j} \langle s_i,s_j\rangle_{FS(H)} \;\chi \wedge c_1(FS(H))^{n-1},$$
where $\{s_i\}$ is an orthonormal basis of holomorphic sections of $L_1^k$ with respect to $H$.
Thus a J-balanced metric $H$ is a critical point of the functional $I_{\mu_{k,\chi}^0}$. 
\medskip

The functional $I_{\mu_{k,\chi}^0}$ is the integral of the moment map $\mu_{k,\chi}^0$ (or Kempf-Ness function), in the sense of \cite{MR}. In particular it is decreasing along the rescaled J-balancing flow. Furthermore, due to Kempf-Ness theory and its convexity, its properness on $SL(N+1)$ is equivalent to the existence of a (unique) critical point which turns out to be a  J-balanced metric (see \cite[Proposition 3.5]{MR} and \cite[Sections 4 and 8]{GRS}, \cite{ThNotes}).
\medskip

One can ask at that stage what is the analogue of $I_{\mu_{k,\chi}^0}$ for the infinite dimensional space of K\"ahler potentials. Let us consider $\omega,\omega_\phi=\omega+\ddbar \phi$ two K\"ahler metrics in $c_1(L_1)$. We define the functional \label{ImuJ} $$I_{\mu_J}(\omega,\omega_\phi)=\int_0^1 \int_M \dot{\phi_t}\left( \frac{1}{\gamma}\chi\wedge \omega_{\phi_t}^{n-1}- \omega_{\phi_t}^n\right)dt,$$ for $\omega_{\phi_t}$ a smooth K\"ahler path from $\omega$ to  $\omega_\phi$. The functional $I_{\mu_J}$ is well defined and independent of the chosen path. Remark that this functional appeared also  in \cite{SW} where it is called $\hat{J}$. Moreover  $I_{\mu_J}(\omega,.)$ vanishes at $\omega$. Without loss of generality, as it is defined up to a constant, we can assume that $J_\chi$ also vanishes at $\omega$.

\begin{lemma}\label{Chenconv}
 The functional $J_\chi$ is strictly convex on the $\C^{1,1}$ geodesics of the space ${\rm Met}(L_1)$ of K\"ahler potentials in $c_1(L_1)$.
\end{lemma}
\begin{proof}
 See \cite[Proposition 2.1]{C1}.
\end{proof}

We sum up the main properties of $I_{\mu_J}$ in the next proposition.
\begin{proposition}
 The functional $I_{\mu_J}$ is strictly convex on the $\C^{1,1}$ geodesics of the space ${\rm Met}(L_1)$ of K\"ahler potentials in $c_1(L_1)$.  Along Donaldson's J-flow,  the functionals $I_{\mu_J}$ and $J_\chi$ are equal and decreasing.   The functionals $I_{\mu_J}$ satisfies the cocyclicity property $$I_{\mu_J}(\omega,\omega_{ \phi_0})+I_{\mu_J}(\omega_{ \phi_0},\omega_{ \phi_1})=I_{\mu_J}(\omega,\omega_{ \phi_1})$$
for $\omega_{ \phi_0},\omega_{ \phi_1}$ K\"ahler forms in the K\"ahler class $[\omega]$. $I_{\mu_J}$ is the integral of the moment map of $\mu_J$ defined by \eqref{mu1}.
\end{proposition}
\begin{proof}
 It is well known that the functional $$I^{\text{AYM}}(\omega,\omega_\phi)=-\int_0^1 \int_M \dot{\phi}_t \omega_{\phi_t}^n \, dt$$ (often called Aubin-Yau-Mabuchi energy) is affine along geodesics of the space of K\"ahler potentials. Therefore the convexity is just a consequence of Lemma \ref{Chenconv}. Moreover,  $I_{\mu_J}$ and $J_\chi$ are equal since $\int_M \dot{\phi_t} \omega_{\phi_t}^n$ vanishes and they agree at $\omega$. Furthermore, along the flow, they are decreasing by definition. The cocyclicity property can be proved following the lines of \cite[Theorem 2.3]{Ma1}.
\end{proof}

\begin{lemma}
 The functional $J_\chi\circ FS$ is convex along geodesics of $\mathcal{B}_k$.
\end{lemma}
\begin{proof}
Firslty, let us explain formally why the result holds. A result of Phong-Sturm asserts that the geodesics in the the space of K\"ahler potentials can be approximated by 
the image via the $FS$ map of geodesics from the Bergman space $\mathcal{B}_k$ (see \cite{P-S1,Bern,BerKel}). Now, from Lemma \ref{Chenconv}, the functional $J_\chi$ is strictly convex on the geodesics in the space of K\"ahler potentials which leads to the convexity of $J_\chi\circ FS$ along geodesics of $\mathcal{B}_k$. 
\par We proceed by a direct computation to prove the lemma. A geodesic in $\mathcal{B}_k$ is just a line. Given $A$ a Hermitian matrix and a Hamiltonian function $H_A=\tr(A\mu_{FS})$  for the corresponding action and the 1-parameter group of embeddings $\iota_t=t^A\circ \iota$, one needs to evaluate the derivative with respect to $t$ of the quantity
$$\frac{1}{\gamma}\int_M \iota_t^*(H_A) \chi \wedge \iota_t^*(\omega_{FS}^{n-1}),$$
but this is equal, up to the factor $\frac{1}{\gamma}$, to 
\begin{align*}
=&\int_M \vert\nabla h_A\vert^2 \chi \wedge \omega^{n-1}_{FS} -2\int_M h_A \partial\bar\partial h_A \wedge \chi\wedge \omega_{FS}^{n-1} ,\\
=&\int_M \vert\nabla h_A\vert^2 \chi \wedge \omega^{n-1}_{FS} - 2\int_M \partial h_A \wedge \bar\partial h_A \wedge \chi\wedge \omega_{FS}^{n-1}. 
\end{align*}
Now, the last term is 
\begin{align*}
\int_M \partial h_A \wedge \bar\partial h_A \wedge \chi\wedge \omega_{FS}^{n-1} =&  \int_M \vert \partial h_A\vert^2_{TM} \;\chi \wedge \omega^{n-1}_{FS}- \frac{1}{n} \int_M \vert \partial h_A\vert^2_{TM,\chi} \; \omega_{FS}^{n}.
\end{align*}
Here $\vert \partial h_A\vert^2_{TM}=\frac{1}{2}\vert \nabla h_A\vert^2_{TM}$ is the norm of the tangential part to $M$ of $\partial h_A$  and $\vert \partial h_A\vert^2_{TM,\chi}$ is the norm with respect to $\chi$.  
Consequently, the derivative we are looking for is given by
\begin{align*}\int_M \vert\nabla h_A\vert^2_{\perp} \; \chi\wedge \omega_{FS}^{n-1} + \frac{1}{n} \int_M \vert\partial h_A\vert^2_{TM,\chi} \;  \omega_{FS}^{n}\geq  0,
\end{align*}
where $\vert\nabla h_A\vert^2_{\perp}$ stands for the norm of the normal component.
%In the computation we have decomposed the restriction of the vectors to $M$ in two components: the component which is tangent to $M$ plus the component which is perpendicular with respect to the obvious metrics. 
Therefore, we have obtained the required convexity.
\end{proof}

\begin{corollary}
 The functional $I_{\mu_J}\circ FS$ is convex along geodesics of $\mathcal{B}_k$.
\end{corollary}
\begin{proof}
This is a consequence of the previous lemma and the fact that the functional $I^{\text{AYM}}\circ FS$ is convex along geodesics in the Bergman space, cf. \cite[Proposition 1]{D2}.
\end{proof}

%We have that for $h_t=FS(H_t)=e^{\phi_t}h_0$,
%$$\frac{d^2 J_\chi\circ FS(h_t)}{dt^2}=\int_M \left(\dot{\phi_t}\Delta{\dot{\phi_t}} + \ddot{\phi_t}\right) \chi \wedge c_1(h_t)^{n-1}- \int_M \dot{\phi_t}\tilde{\Delta}{\dot{\phi_t}} \; c_1(h_t)^n$$
%The last integral is non negative. For a Bergman geodesic in $\mathcal{B}_k$ written as
%$$\phi_t=-\log \sum_{i=1}^{N+1} e^{\alpha_i t}|S_i|^2_{h_0},$$
%using orthonormal sections (a geodesic in $\mathcal{B}_k$ is just a line), the reasoning of \cite[Proposition 1]{D2} shows that the pointwise quantity $\dot{\phi_t}\Delta{\dot{\phi_t}} + \ddot{\phi_t}$ is non-negative.

Now, using the fact that $\log\det$ is linear on geodesics, we also get
\begin{corollary}\label{uniquenessJbal}
The functional $I_{\mu_{k,\chi}^0}$ is convex along geodesics of $\mathcal{B}_k$. It has at most one critical point. 
A J-balanced metric is an absolute minimum of the functional $I_{\mu_{k,\chi}^0}$ and is in particular unique.
\end{corollary}

%Consider the function on $GL(N_k+1,\mathbb{C})/U(N)$ for $z\in \mathbb{C}^{N_k+1}$
%$$R_z(H)=\log |z|^2_{H^{-1}} + \frac{1}{N+1}\log \det H$$
%Then $R_z$ is convex on all geodesics. Consider now the functional on the Bergman space
%$$F(H)=\int_{\mathbb{P}^{N_k}} R_z(H) \frac{1}{k^{n-1}}\chi \wedge c_1(FS(H))^{n-1}$$

\subsection{Iterates of the maps $Hilb_\chi \circ FS$ and $FS\circ Hilb_\chi$}
%\texorpdfstring{$Hilb_\chi \circ FS$}{HilboFS}
%\texorpdfstring{$FS\circ Hilb_\chi$}{FSoHilb}
In this section, we investigate the iterates of the map $T_{k,\chi}$.

\begin{lemma}\label{lemma1}
 Consider $h_0\in \Met(L_1),\; h=e^{-\phi}h_0\in \Met(L_1)$. Then
$$\frac{1}{\gamma}\int_M \phi \; \chi\wedge c_1(h)^{n-1}\leq J_\chi(h)-J_\chi (h_0) \leq \frac{1}{\gamma}\int_M \phi \; \chi\wedge c_1(h_0)^{n-1}.$$ 
\end{lemma}
\begin{proof}
 If one defines $h_t=e^{-t\phi}h_0$,  and  $$f(t)=\frac{1}{\gamma}\int_M t \phi \chi\wedge c_1(h_0)^{n-1} - \left(J_\chi(h_t)-J_\chi(h_0)\right),$$
then $f(0)=f'(0)=0$ and furthermore along the considered path, 
\begin{align*}
f''(t)&=-(n-1)\frac{1}{\gamma}\int_M \phi \chi\wedge c_1(h_t)^{n-2}\wedge \ddbar \phi , \\
&=(n-1)\sqrt{-1}\frac{1}{\gamma}\int_M \partial \phi \wedge \bar{\partial} \phi \wedge \chi\wedge c_1(h_t)^{n-2}, 
\end{align*}
which is non-negative. Thus $f(t)\geq 0$ at $t=1$ which provides one inequality. By symmetry, we get the result. One can perform a direct computation. For instance, when $n=2$, $J_\chi(h)-J_\chi (h_0)$ can be written as 
$\frac{1}{2\gamma}\int_M \phi \,\chi\wedge (c_1(h_0)+c_1(h))$.
\end{proof}

Define for $h\in \Met(L_1^k)$, $H\in \Met(H^0(L_1^{k}))$,
$$\hat{P}(h,H)=\log \sum_{i=1}^{N+1} \Vert S_i\Vert^2_{Hilb_\chi(h)}- \log (N+1) + \log \det(H)+\frac{N+1}{\Vol_{L_1}(M)} J_\chi(h)$$
where $\{S_i\}$ form an orthonormal basis with respect to $H$. Then it is not difficult to check that
$$\hat{P}(FS(H),H)=\frac{N+1}{\Vol_{L_1}(M)}I_{\mu_{\chi,k}^0}(H).$$

\begin{lemma}\label{decreas1}
For any metrics $h,H$, one has
$$\hat{P}(h,H)\geq \hat{P}(FS(H),H).$$
\end{lemma}
\begin{proof}
 One checks that if we define $h=e^{-\phi}FS(H)$, then
\begin{eqnarray*}\hat{P}(h,H)-\hat{P}(FS(H),H) &=& \log\left( \frac{N+1}{\gamma \Vol_{L_1}(M)}\int_M e^{-\phi} \chi \wedge c_1(h)^{n-1}\right) \\
&&+\frac{N+1}{\Vol_{L_1}(M)}\left(J_\chi(h)-J_\chi(FS(H))\right), \\
&\geq &  \frac{N+1}{\gamma\Vol_{L_1}(M)}\int_M -\phi \;\chi\wedge c_1(h)^{n-1}\\
&&+\frac{N+1}{\Vol_{L_1}(M)}\left(J_\chi(h)-J_\chi(FS(H))\right),\\
&\geq& 0,
 \end{eqnarray*}
 using Lemma \ref{lemma1}. 
\end{proof}

\begin{lemma}\label{decreas2}
 $$\hat{P}(h,H)\geq \hat{P}(h,Hilb_\chi(h)).$$
\end{lemma}
\begin{proof}
This is a consequence of the arithmetic-geometric inequality.
\end{proof}

Assume the existence of $h_{bal}\in \Met(L_1^k)$ J-balanced  metric and $H_{bal}\in \Met(H^0(L_1^k))$ J-balanced metric
on the Bergman space. Then for any $h\in \Met(L_1^k),$ and $H\in \Met(H^0(L_1^k))$ one has
\begin{eqnarray*}
\hat{P}(h,H)&\geq & \hat{P}(FS(H),H), \\
                       &=&\frac{1}{\Vol_{L_1}(M)} I_{\mu_{\chi,k}^0}(H),\\
                      & \geq &\frac{1}{\Vol_{L_1}(M)} I_{\mu_{\chi,k}^0}(H_{bal}),\\
                       & = &   \hat{P}(FS(H_{bal}),H_{bal}),\\
                     &=&  \hat{P}(h_{bal},H_{bal}).
\end{eqnarray*}
In particular it gives that $$I_{\mu_{\chi,k}^0}(H)\geq I_{\mu_{\chi,k}^0}(H_{bal}).$$ 
Moreover  $\hat{P}(h,Hilb_\chi(h))\geq \hat{P}(h_{bal},Hilb_\chi(h_{bal}))$. \\
Thus the functional on $\Met(L_1^k)$ defined by\label{hatI}
$$\hat{I}_k(h):=\frac{\Vol_{L_1}(M)}{N+1}\hat{P}(h,Hilb_\chi(h))=J_\chi(h)+\frac{\Vol_{L_1}(M)}{N+1}\log \det Hilb_\chi(h)$$
satisfies
$$\hat{I}_k(h)\geq \hat{I}_k(h_{bal}).$$
We will see soon that this new functional has a geometric interpretation. A direct consequence of Lemmas \ref{decreas1} and \ref{decreas2} is the following corollary.
\begin{corollary}In our setting, the  following hold:
\begin{enumerate}
 \item A J-balanced metric on $\Met(L_1^k)$ (resp. $\Met(H^0(L_1^k)$) is a minimum of the functional $\hat{I}_k$ (resp. $I_{\mu_{\chi,k}^0}$). 
\item The map $FS \circ Hilb_\chi$ decreases the functional $\hat{I}_k$ while the map $Hilb_\chi \circ FS$ decreases the functional $I_{\mu_{\chi,k}^0}$.
\item The functional $\hat{I}_k$ is bounded from below if and only if  the functional $I_{\mu_{\chi,k}^0}$ is bounded
from below.
\end{enumerate}

\end{corollary}

We now explain the asymptotic behavior of the functional $\hat{I}_k$, by studying the term $$\log \det Hilb_\chi(h_t)$$ where 
$h_t=he^{-k\phi_t}$ (with $\Vert\phi_t\Vert_{\C^{\infty}}=O(1)$ when $k\rightarrow \infty$) is a path in $\Met(L_1^k)$. Then, we can write $Hilb_\chi(h_t)=\langle s_i,s_j\rangle_{Hilb(h_t)}$ where $\{s_i\}$ is an orthonormal basis with respect to $Hilb_\chi(h)$. Thus its derivative at $t=0$ is given by the derivative of $\sum_i \Vert s_i \Vert^2_{Hilb_\chi(h_t)}$ and because of the variation of the volume form, this is written as 
\begin{align*}
-\frac{1}{\gamma}\int_M \hspace{-0.07cm}k\dot{\phi} \sum_{i} \vert s_i\vert^2_h \, \chi\hspace{-0.07cm}\wedge \hspace{-0.07cm} c_1(h)^{n-1}\hspace{-0.07cm}+\hspace{-0.07cm} (n\hspace{-0.07cm}-\hspace{-0.07cm}1)\hspace{-0.07cm}  \sum_{i} \vert s_i\vert^2_h \,\chi\hspace{-0.07cm}\wedge\hspace{-0.07cm} c_1(h)^{n-1}\hspace{-0.07cm} \wedge\hspace{-0.07cm} \ddbar \dot{\phi}, 
%&=-\frac{1}{\gamma}\int_M \dot{\phi}\left(  k\sum_{i} \vert s_i\vert^2_h \, \chi\wedge c_1(h)^{n-1} + (n-1)\ddbar  \sum_{i} \vert s_i\vert^2_h \wedge   \chi\wedge c_1(h)^{n-1}\right).
\end{align*}
Together with \eqref{bergJflow} and the fact that the second term in the integrand is negligible compared to the first one when $k\rightarrow +\infty$ (by uniformity of the Bergman expansion in $\C^2$ topology), we obtain when $k\rightarrow + \infty$ that
$$\frac{d}{dt}_{\vert t=0} \left( \frac{\Vol_{L_1}(M)}{N+1}\log \det Hilb_\chi (he^{-k\phi_t}) \right)=-k\int_M \dot{\phi}\, c_1(h)^n + O(1),$$
where we have used that $N=\Vol_{L_1}(M)k^n+ O(k^{n-1})$. This leads to the following conclusion for $\hat{I}_k$.

\begin{corollary}
 Over compact subsets of $\Met(L_1)$, the functionals $\hat{I}_k$ and $I_{\mu_J}$ are equivalent, up to a normalisation, i.e 
$$\frac{1}{k}\hat{I}_k(h^k) = I_{\mu_J}(h) + O(1/k)$$
for $h\in \Met(L_1)$.
\end{corollary}
Corollary \ref{cor1} and our last result show  that a critical metric solution of \eqref{critical} is actually an absolute minimum of $I_{\mu_J}$.  Of course this is conceptually a consequence of  $I_{\mu_J}$ being the integral of the moment map ${\mu_J}$. Formally, this can be seen from the facts that the J-flow decreases $I_{\mu_J}$, and from the analytic study of the flow (convergence to a unique critical metric in smooth topology).
% (see \cite{LiChi} for the CSCK case

\medskip

Finally, following the techniques of \cite{Sa}, we obtain that if there is a J-balanced metrics of order $k$, then
the iterates of $Hilb_\chi\circ FS$ on $\Met(H^0(L^k_1))$ will converge towards this metric. 
\begin{theorem}\label{conviter}
 Assume that there exists $H_\infty\in \Met(H^0(L^k_1))$ J-balanced. For any $H_0\in \Met(H^0(L^k_1))$,
denote from now $$H_l=Hilb_\chi\circ FS(H_{l-1}),$$ for $l\geq 1$. Then, up to a positive constant $r$, 
$$H_l\rightarrow rH_\infty$$
as $l\rightarrow +\infty$.
\end{theorem}
For the sake of clarity, we give the details of the proof which consists in an easy modification of \cite{Sa,Sey2}, which is not surprising since it is a purely finite dimensional problem and our setting is very close. We will decompose the proof into several lemmas starting with the following definition.
\begin{definition}
Let $\{s_i\}$ be a basis of $H^0(L^k_1)$. Using this basis, we can view elements of $\Met(H^0(L^k_1))$ as Hermitian matrices $(N+1)\times (N+1)$. 
A subset $\mathcal{U}\subset\Met(H^0(L^k_1))$ is bounded if there exists a number $R>1$ satisfying the following. For any $H\in \mathcal{U}$, there exists a constant $\gamma_H>0$ so that the smallest and largest eigenvalues of $H$ satisfy
$$\frac{\gamma_H}{R}\leq \min \frac{\vert H(\zeta)\vert}{\vert \zeta\vert} \leq \max \frac{\vert H(\zeta)\vert}{\vert \zeta \vert} \leq \gamma_H R.$$
\end{definition}

With the notations of the the previous definition, we have the following obvious proposition due to the fact that the closure of bounded sets are compact in finite dimension.
\begin{proposition}
Any bounded sequence $H_k$ has a subsequence $H_{n_k}$ such that $\frac{1}{\gamma_{n_k}}H_{n_k}$ converges in $\Met(H^0(L^k_1))$. 
\end{proposition}

\begin{lemma}\cite[Lemma 3.2]{Sey2}\label{le6}
 The set $\mathcal{U}$ is bounded if and only if there exists a number $R>1$ so that for any $H\in\mathcal{U}$, we have
$$\frac{1}{R}\leq \min \frac{\vert \tilde{H}(\zeta)\vert}{\vert \zeta\vert} \leq \max \frac{\vert \tilde{H}(\zeta)\vert}{\vert \zeta \vert} \leq R,$$
where $\tilde{H}=\frac{1}{\det(H)^{\frac{1}{N+1}}}H$.
\end{lemma}
\begin{proof}
We include the proof for the sake of clarity. Without loss of generality, we can assume $H(s_i,s_j)$ is diagonal with entries $e^{\lambda_i}$, $\lambda_1 \leq .. \leq \lambda_{N+1}$. Assuming $\mathcal{U}$ bounded, we obtain $\gamma_H \leq Re^{\lambda_i}$ and $\gamma_H \geq \frac{1}{R}e^{\lambda_i}$. Thus, $e^{\lambda_{N+1}}\leq R^{2}e^{\lambda_i}$ and $e^{\lambda_1}\geq R^{-2}e^{\lambda_i}$ for all $i=1,..,N+1$,  which gives $$\det(H)^{-1/(N+1)}e^{\lambda_{N+1}}=\left(\prod_i e^{\lambda_{N+1}-\lambda_i}\right)^{1/(N+1)}\leq R^2.$$ Similarly, we obtain $\det(H)^{-1/(N+1)}e^{\lambda_{1}}\geq \frac{1}{R^2}$.
\end{proof}

\begin{lemma}\label{le8}
 Under the assumption of the theorem, if the sequence $H_l$ is bounded in $\Met(H^0(L_1^k))$, then the sequence $\det(H_l)$ is convergent and $$\det(H_{l+1}H_l^{-1})\rightarrow 1$$ as $l\rightarrow +\infty$.
\end{lemma}
\begin{proof}
 From Lemma \ref{decreas2}, we deduce that the sequence $\log \det (H_l)$ is decreasing. From Lemma \ref{decreas1}, we deduce that the sequence $J_\chi \circ FS(H_l)$  is also decreasing. Since $I_{\mu_{\chi,k}^0}(H_l)$ is decreasing and bounded from below, $\log \det (H_l)$ is bounded and converges. 
\end{proof}

\begin{lemma}\label{le7}
 Assume the sequence $H_l$ is bounded in $\Met(H^0(L_1^k))$. Let $H\in \Met(H^0(L_1^k))$ and $\{s_i^{l}\}_i$ be an orthonormal basis with respect to $H_l$ so that the matrix $H(s_i^l, s_j^l)$ is diagonal. Then, $$\lim_{l\rightarrow +\infty}\Vert s_i^l\Vert^2_{Hilb_\chi(FS(H_l))} =\frac{\Vol_{L_1}(M)}{N+1}.$$
\end{lemma}
\begin{proof}
 Let us consider $\hat{s}_i^l$ another basis, orthonormal with respect to $H_l$ and so that $H_{l+1}(\hat{s}_i^l,\hat{s}_j^l)$ is diagonal.
From the previous lemma, we deduce that $\lim_{l\rightarrow +\infty}\det(H_{l+1}(\hat{s}_i^l,\hat{s}_i^l))=1$. 
We have always that $$\tr(Hilb_\chi(FS(H))H^{-1})=N+1,$$ so we get $\tr(H_{l+1}(\hat{s}_i^l,\hat{s}_i^l))=N+1$ for all $l$. It is not difficult to check that the arithmetic-geometric inequality implies $$H_{l+1}(\hat{s}_i^l,\hat{s}_i^l)\rightarrow 1$$ as $l\rightarrow +\infty$. Now, we write $$s_i^l=\sum_{j=1}^{N+1} a_{ij}^l\hat{s}_j^l,$$
with $(a_{ij}^l)\in U(N+1)$. The matrix $a_{ij}^l$ converges when $l\rightarrow +\infty$ up to taking a subsequence. For the limit $(a^{\infty}_{ij})\in U(N+1)$ we have $$H_{l+1}(s^l_i,s^l_j)\rightarrow \sum_{j=1}^{N+1} \vert a_{ij}^{\infty}\vert^2 =1,$$
which means that 
$$\lim_{l\rightarrow +\infty}\frac{N+1}{\gamma\Vol_{L_1}}\int_M \vert s_i^l\vert^2_{FS(H_l)}\, \chi\wedge c_1(FS(H_l))^{n-1}=1,$$
as expected. 
\end{proof} 

\begin{lemma}\label{le9}
If the sequence $H_l$ is bounded, then for any Hermitian metric $H\in \Met(H^0(L_1^k))$ and $\epsilon>0$, we have
\begin{equation}\label{smallineq} 
I_{\mu_{\chi,k}^0}(H)\geq I_{\mu_{\chi,k}^0}(H_l)- \epsilon
\end{equation}
for $l$ sufficiently large.
\end{lemma}
\begin{proof} 
Fix $\{s_i^l\}$ an orthonormal basis with respect to $H_l$ such that $H(s_i^{l},s_j^{l})$ is diagonal with entries $e^{\lambda_i^l}$. Define $f(t)=I_{\mu_{\chi,k}^0}(H_t)$ where $H_t$ is the matrix of entries $e^{t\lambda_i^l}$, so that $H_{\vert t=0}=H_l$ and $H_{\vert t=1}=H$. By convexity of $I_{\mu_{\chi,k}^0}$ in the Bergman space, we have $f(1)-f(0)\geq f'(0)$. By definition, one gets
\begin{align*}
f'(0)&=\int_M \frac{d}{dt}_{\vert t=0}(FS(H_t)) \frac{\chi\wedge c_1(FS(H_t))^n}{\gamma} - \frac{\Vol_{L_1}(M)}{N+1}\sum_i {\lambda_i}^l, \\
&=  \int_M  \sum_{i=1}^{N+1} \lambda_i^l \vert s_i^l \vert^2 _{FS(H_l)}  \frac{\chi\wedge c_1(FS(H_t))^n}{\gamma} - \frac{\Vol_{L_1}(M)}{N+1}\sum_i {\lambda_i}^l.
\end{align*} 
Let us assume that $\lambda_i^l$ are bounded. Then we can apply Lemma \ref{le7} and obtain that $f'(0)\rightarrow 0$ when $l\rightarrow +\infty$, which provides 
eventually \eqref{smallineq}.\\
It remains to show that $\lambda_i^l$ are all bounded when $l$ varies. Using the fact that $\{e^{-\lambda_i^l/2}s_i^l\}_i$ is an orthonormal basis for $H$ and Lemma \ref{le6}, we have the existence of $R>1$ such that $$\frac{1}{R}< \frac{H_l(s_i^l,s_i^l)e^{-\lambda_i}}{\det(H_l)^{1/(N+1)}}=\frac{e^{-\lambda_i}}{\det(H_l)^{1/(N+1)}} < R $$
from which we deduce that $\lambda_i^l$ is bounded if and only if $\det(H_l)$ is bounded. But this is the case by Lemma \ref{le8}.
\end{proof}

\begin{proof}[Proof of Theorem \ref{conviter}]
As we have already seen in the previous section, the J-balanced metric $H_\infty$ is unique up to normalisation. To normalise our metrics, we choose to work in the space $$\{ \tilde{H}=(\det H)^{-1/(N+1)}H, \, H\in {\Met}(H^0(L_1^k))\}\subset {\Met}(H^0(L_1^k)).$$
Furthermore, as integral of the moment map $\mu_{\chi,k}^{0}$, the functional $I_{\mu_{\chi,k}^0}$ is proper and bounded from below. 
The sequence  $I_{\mu_{\chi,k}^0}(H_l)$ is decreasing, and thus $J_\chi((\det H_l)^{-1/(N+1)}FS(H_l))$ is also bounded. By properness, it follows that $$\tilde{H}_l=(\det H_l)^{-1/(N+1)} FS(H_l)$$ is bounded. This forces this sequence to converge as we now see. 
\par Suppose not. Then we can at least take a non convergent  subsequence $\tilde{H}_{l_{k}}$ which always remain at a distance $\epsilon$ of the J-balanced metric $H_\infty$. But $\tilde{H}_{l_{k}}$ is bounded and its image by  $I_{\mu_{\chi,k}^0}$ converges to the minimum of the functional $I_{\mu_{\chi,k}^0}$, up to taking a subsequence that we denote $\tilde{H}_{l_{k_m}}$. In fact, we have obtained from the previous results that for any bounded sequence $H_l$, $I_{\mu_{\chi,k}^0}(H_l)$ converges to the minimum of the functional $I_{\mu_{\chi,k}^0}$, see Lemma \ref{le9}. Therefore, $\tilde{H}_{l_{k_m}}$ converges and its limit is actually a J-balanced metric from Corollary \ref{uniquenessJbal}. This is a contradiction with the fact that all the terms $\tilde{H}_{l_k}$ are at distance $\epsilon$ of $H_\infty$.
\par From Lemma \ref{le8}, we get that $\log(\det(H_l))$ is bounded and decreasing. This allows us to conclude that $H_l$ is convergent to $rH_\infty$. 
\end{proof}

From Corollary \ref{cor1} we obtain the following result.
\begin{corollary}\label{Jcoroiter}
 Assume the existence of a critical metric solution $\omega_{\infty}$ of \eqref{Jflow}. Then for all $k\gg 0$,  the map
$$Hilb_\chi \circ FS:Met(H^0(L^k_1))\rightarrow \Met(H^0(L^k_1))$$ (respectively $T_{k,\chi}=FS\circ Hilb_\chi:\Met(L_1^k) \rightarrow \Met(L_1^k)$) defines a dynamical system that has a fixed (unique)
attractive point, the J-balanced metric $H_{bal}$ (respectively $h_{bal}$ with $FS(H_{bal})=h_{bal}$). Furthermore,
$$\Big\Vert \omega_{\infty}-\frac{1}{k}c_1(FS(H_{bal}))\Big\Vert_{\C^\infty}=O\left(\frac{1}{k}\right).$$ 
\end{corollary}

\section{Algebro-geometric aspects}\label{sectalg-geom}

In this section we describe a notion of \emph{Chow stability} for a linear system $|L_2|$ in a fixed polarised variety $(M,L_1)$, which is related to Geometric Invariant Theory (GIT). Our main result is then that the existence of a J-balanced metric implies this notion of Chow stability for $|L_2|$. It follows that the existence of a critical point of the J-flow implies asymptotic Chow stability of $|L_2|$. Our definition of asymptotic Chow stability is motivated by Lejmi-Sz\'ekelyhidi's notion of J-stability (see Definition \ref{jstabilityone}), which is an analogue of K-stability for the J-flow. Indeed we show that asymptotic Chow semistability of $|L_2|$ implies J-semistability, just as asymptotic Chow semistability of a polarised variety implies K-semistability. 

Before discussing the notions of stability relevant to us, we recall some standard GIT as motivation. Let $G\subset SL(n+1,\mathbb{C})$ act on a scheme $M\subset\pr^n$, and let $x\in M$ have lift $\hat{x}\in \comp^{n+1}$. We say the point $x$ is \emph{stable} in the sense of GIT if the orbit $G.\hat{x}$ is closed and $x$ has finite stabiliser. There is a weaker notion of polystability which allows higher dimensional stabiliser, but we will not need this. Roughly speaking, a quotient $M \sslash G$ parameterising polystable orbits exists. To understand whether or not the point is represented in the quotient, one therefore seeks to understand the \emph{stability} of the point.

The Hilbert-Mumford criterion gives a numerical way of checking this stability. Let $\lambda:\comp^*\hookrightarrow G$ be a one-parameter subgroup. Denote the limit $x_0=\lim_{t\to 0} \lambda(t).x$, so that $x_0$ is a fixed point of $\lambda$. There is therefore a $\comp^*$-action on the line above $x_0$, which acts with weight $w=-\mu(\lambda,x).$ By this we mean that, for $\hat{x}_0$ any lift of $x_0$, we have $\lambda(t).\hat{x}_0 = t^w\hat{x}_0.$ Then the Hilbert-Mumford criterion states that $x$ is stable if and only if $\mu(x,\lambda)>0$ for all one-parameter subgroups $\lambda$. We will later need the following properties of the weight where one fixes a one-parameter subgroup $\lambda$ and varies $x$.

\begin{lemma}\label{generalweight}Let $M$ be an irreducible subvariety of $\pr^n$. Fix a one-parameter subgroup $\lambda:\comp^*\hookrightarrow G$. Then the weight $\mu(\lambda,x)$ is constant outside a Zariski closed subset of $M$. This general value is equal to the largest weight of any point in $X$. \end{lemma}

\begin{proof} Both of these properties are consequences of the following alternative characterisation of the weight of a one-parameter subgroup. Diagonalise the $\comp^*$-action as $\diag(t^{\lambda_0}\hdots t^{\lambda_n})$ with $\lambda_0\leq \hdots\leq \lambda_n$, and set $l=\min \{i \ | \ x_i\neq 0\}.$ Then $\mu(\lambda,x) = -\lambda_l.$ Remark that $-\lambda_l \geq -\lambda_{l+1}\geq\hdots\geq-\lambda_{n}$.

We now use the hypothesis that $M$ is irreducible. The intersection of a hyperplane with the irreducible variety $X$ is either empty, or of codimension one in $M$. The subset of $M$ with weight \emph{not equal} to $\lambda_0$ is the intersection of $M$ with the hyperplane $\{x_0=0\}$; by irreducibility this is either empty or of codimension one. In the latter case, the general point of $M$ has weight $\lambda_0$ and we are done. On the other hand if it is empty, we continue this process until we find the smallest $i$ such that $X\cap \{x_i=0\}\neq \varnothing$ (remark that $M$ must intersect some hyperplane section). Then the general point of $M$ will have weight $-\lambda_i$. By construction, $-\lambda_i$ is the largest weight of any point in $M$.
\end{proof}

\begin{remark} It is essential here that $M$ is irreducible; if $M=V(xy)\subset\comp^2$ and $\lambda$ is given as $\diag(t^a,t^b)$, then clearly no such general weight exists for $a\neq b$. \end{remark}

One can apply GIT to \emph{moduli spaces} of varieties. Our discussion roughly follows \cite{RT2}. Let $Y\subset\pr^n$ be a fixed subvariety. There are two natural moduli spaces in which one can naturally consider $Y$ as a point. The first is the Hilbert scheme of subschemes of $\pr^n$ with the same Hilbert polynomial, which essentially parameterises ideal sheaves. The second is the Chow variety, whose definition uses intersection theory and compactifies by adding cycles at the boundary. Both moduli spaces naturally embed in certain Grassmanians constructed from $\pr^n$, and so we will have corresponding GIT problems. To calculate stability of points in each moduli space, we will need to understand the line above a point under these embeddings into Grassmanians.

We first consider the Hilbert scheme, for which we refer to \cite[Theorem 1.4]{JK-rational} for the detailed construction. For our purposes we only need to understand how one obtains a point in a certain projective space from a given scheme, thus our construction omits many important details. The exact sequence of sheaves on $\pr^{n}$ $$0\to \scI_Y\to \scO_{\pr^{n}}\to \scO_Y\to 0$$ induces by twisting with $\scO_{\pr^{n}}(1)$ an exact sequence for sufficiently large $K$ \begin{equation}\label{hilbpoint}0\to H^0(\pr^{n},\scI_Y(K))\to S^KH^0(\pr^n,\scO(1))\to H^0(Y,\scO_Y(K))\to 0.\end{equation} If one fixes a Hilbert polynomial, one can choose $K$ \emph{independently} of $Y\subset \pr^{n}$ (see for example \cite[Lecture 14]{DM-lectures} or \cite[Theorem 1.5]{JK-rational}); in particular the previous sequence is exact for \emph{all} such $Y$.\\

The Hilbert scheme of subschemes of $\pr^{n}$ with Hilbert polynomial $p(K)=\dim H^0(Y,\scO_Y(K))$ is embedded via this exact sequence as a subscheme of $\Grass(S^KH^0(\pr^n,\scO(1)), p(K))$, with the point corresponding to $Y$ given by the vector subspace $H^0(\pr^{n},\scI_Y(K))$. In turn these Grassmanians naturally embed into projective spaces using the Pl\"ucker embedding, explicitly we have embeddings $$\Hilb\hookrightarrow \Grass(S^KH^0(\pr^n,\scO(1)), p(K)) \hookrightarrow \pr(\Lambda^{p(K)}(S^KH^0(\pr^n,\scO(1)))).$$ The $SL(n+1,\comp)$ action on $\pr^n$ naturally induces an action on these projective spaces, and so for each $K$ we have an associated GIT problem.

\begin{definition} We say that $Y\subset\pr^n$ is \emph{Hilbert stable} if the corresponding point in the Hilbert scheme is GIT stable embedded as a subscheme of projective space using the Pl\"ucker embedding for $K\gg 0$.\end{definition}

We now turn to the Chow variety, for which we refer to \cite[Section 1.3]{JK-rational} for the detailed construction and definition. Again we only need to understand how, in the construction of the Chow variety, one associates a point in a certain projective space from a given scheme, so we omit many important details in our discussion. We denote by $m$ the dimension of $Y\subset\pr^n=\pr(V)$ and $d$ the degree of $Y$, i.e. $(\scO_{\pr^n}(1)|_Y)^m=d$. Let $Z$ be the set of $(n-m-1)$-dimensional planes intersecting $Y$ nontrivially, so that $Z\subset \Grass(n-m,n+1)$. We denote the Pl\"ucker embedding of this Grassmanian as \begin{equation}\label{plucker}Pl: \Grass(n-m,n+1) \hookrightarrow \pr(\Lambda^{n-m}V).\end{equation} One can show that $Z$ is of codimension one in the Grassmanian, and that $Z=V(f)$ for some section $f\in H^0(\Grass(n-m,n+1), \scO(d))$ unique up to scaling. One therefore has a corresponding point $[f]\in \pr(H^0(\Grass(n-m,n+1), \scO(d)))$, called the Chow point. The Chow variety has a natural compactification obtained by adding limit cycles at the boundary. Again the $SL(n+1,\comp)$ action on $\pr^n$ induces one on $\pr(H^0(\Grass(n-m,n+1), \scO(d)))$. 

\begin{definition} We say $Y\subset\pr^n$ is \emph{Chow stable} if its Chow point $[f]$ is GIT stable under the induced action of $SL(n+1,\comp)$.\end{definition}

The Hilbert-Mumford criterion states it is enough to show a corresponding weight is positive for each one-parameter subgroup $\lambda:\comp^*\hookrightarrow SL(n+1,\comp)$. Fixing some one-parameter subgroup $\lambda$, the limit $Y_0=\lim_{t\to 0}\lambda(t).Y$ is naturally a point in the same Hilbert scheme and Chow variety as $Y$, which has a corresponding Hilbert-Mumford weight. Here one considers $Y_0$ as the limit scheme in the Hilbert scheme, and as the limit cycle in the Chow variety. Hence $Y$ is Hilbert stable or Chow stable if and only if these weights are strictly positive for each one-parameter subgroup.

Considering instead a polarised variety $(M,L_1)$, one naturally has a sequence of embeddings $M\hookrightarrow\pr(H^0(X,L_1^r)^*)$ for each $r$ sufficiently large. Hence one can define asymptotic Hilbert stability (respectively asymptotic Chow stability) to mean the point corresponding to $M$ is GIT stable in the appropriate Hilbert scheme (respectively Chow variety) for $k$ sufficiently large. We will need an extension of the above discussion to subvarieties of $M$.

\begin{definition}[Twisted stability] Let $Y\subset M$ be a subvariety. We say that $Y$ is $M$-\emph{twisted asymptotically Hilbert stable (respectively Chow stable)} if it is Hilbert stable (respectively Chow stable) under the natural embeddings $Y\hookrightarrow\pr(H^0(M,L_1^r)^*)$ for all $r\gg 0$.\end{definition}

\begin{remark} Clearly setting $Y=M$ recovers the usual definition of asymptotic Chow and Hilbert stability. \end{remark}

In what follows we will explicitly calculate the weight of the one-parameter subgroup $\mu(\lambda,Y)$ in each moduli space using data arising only from the action on $\pr(H^0(M,L_1^r)^*)$ itself. We first need the following definition, which is a geometrisation of the one-parameter subgroups considered above.

\begin{definition} A \emph{test configuration} $(\scX,\scL)$ for a polarised variety $(M,L_1)$ is a variety $\scX$ together with
\begin{itemize} 
\item a proper flat morphism $\pi: \scX \to \comp$,
\item a $\comp^*$-action on $\scX$ covering the natural action on $\comp$,
\item and an equivariant relatively very ample line bundle $\scL$ on $\scX$
\end{itemize}
such that the fibre $(\scX_t,\scL_t)$ over $t\in\comp$ is isomorphic to $(M,L_1^r)$ for one, and hence all, $t \in \comp^*$ and for some $r>0$. We call $r$ the \emph{exponent} of the test configuration. \end{definition}

For each subvariety $Y\subset X$, by taking the closure of $Y$ under the $\comp^*$-action, one naturally obtains a test configuration for $(Y,L_1)$, where we have abused notation by writing $L_1$ as the restriction of $L_1$ to $Y$. We denote this induced test configuration by $(\scY,\scL)$, with similar abuse of notation.

\begin{proposition}\label{onepsandtestconfigs}\cite[Proposition 3.7]{RT2} A test configuration of exponent $r$ is equivalent to a one-parameter subgroup of $GL(h(r),\comp)$, where $h(r)$ is the Hilbert polynomial of $(M,L_1)$. Given a one-parameter subgroup $\lambda$, the corresponding test configuration is $\scX_t = \lambda(t).X$ with the line bundle the restriction of the $\scO(1)$ from the projective space. Conversely, for each $K>0$ one can equivariantly embed a given test configuration $(\scX,\scL)$ of exponent $r$ into $\pr(H^0(X,L^{rK})^*)$ such that the test configuration is realised by some one-parameter subgroup.\end{proposition}

We are now in a position to give a numerical criterion for $M$-twisted asymptotic Hilbert and Chow stability. In order to do so, we introduce the following notation.

Let $(\scX,\scL)$ be a test configuration. As the $\comp^*$-action on $(\scX,\scL)$ fixes the central fibre, there is an induced $\comp^*$-action on $(\scX_0,\scL_0)$ and hence on $H^0(\scX_0,\scL^K_0)$ for each $K$. We denote the Hilbert polynomial and total weight of this action respectively by \begin{align*} h(K)&=a_0K^n+a_1K^{n-1}+O(K^{n-2}), \\ w(K) &=b_0K^{n+1}+b_1K^n+O(K^{n-1}).\end{align*} By asymptotic Riemann-Roch and flatness of the test configuration, we have intersection-theoretic formulas for $a_0,a_1,$ as  $$a_0 = r^n\frac{L^n}{n!}, a_1 = r^{n-1}\frac{K_M.L^{n-1}}{2(n-1)!}.$$ For $(\scY_0,\scL_0)$ denote the corresponding Hilbert and weight polynomials by \begin{align*} \hat{h}(K)&=\hat{a}_0K^m+O(K^{m-1}), \\ \hat{w}(K) &=\hat{b}_0K^{m+1}+O(K^{m}),\end{align*} where $m$ is the dimension of $Y$.

\begin{theorem}\label{numericalstability} Let $(M,L_1)$ be a polarised variety and $Y\subset M$ a subvariety. Then $Y$ is $M$-twisted asymptotically Hilbert stable if and only if, for all $r\gg 0$, each test configuration of exponent $r$ has normalised weight (setting $k=rK$) $$\hat{w}_{r,k}=\hat{w}(k)rh(r)-kw(r)\hat{h}(k)>0$$ for $k\gg 0$. 

The normalised weight $\hat{w}_{r,k}$ is a polynomial in $k$ and $r$ of degree $m+1$ in $k$, write $\hat{w}_{r,k} = \sum_{i=0}^{m+1}e_i(r)k^i.$ Then $Y$ is $M$-twisted asymptotically Chow stable if and only if for all $r\gg 0$ we have $e_{m+1}(r)>0$. \end{theorem}

\begin{proof} 

We first prove with the Hilbert stability statement. The Hilbert-Mumford criterion together with Proposition \ref{onepsandtestconfigs} imply $X$-twisted asymptotic Hilbert stability is equivalent to the asymptotic Hilbert weight of each test configuration being strictly positive, provided it lies in $SL(h(r),\comp)$ rather than $GL(h(r),\comp)$. The proof is in two steps. The first is to modify a fixed test configuration to lie in $SL(h(r),\comp)$, and the second is to explicitly calculate the Hilbert weight asymptotically. 

For clarity we consider a more general situation. Let $\lambda$ be a one-parameter subgroup of $GL(N,\comp)$ diagonalised as $\diag(t^{\lambda_0}\hdots t^{\lambda_N})$, with total weight $w=\sum_i \lambda_i$. If $w=0$, then  $\lambda$ lies in $SL(N,\comp)$; assume this is not the case. Then one gets a new one-parameter subgroup by subtracting $\frac{w}{N}$ from each weight, \emph{provided each} $\lambda_i-\frac{w}{N}$ \emph{remains an integer}. As this is not necessarily the case, we first multiply each $\lambda_i$ by $N$ to produce a new one-parameter subgroup with total weight $wN$, and subtract $w$ from each weight to produce a one-parameter subgroup of $SL(N,\comp)$. 

There are two natural operations one can perform on a test configuration which leave the total space $\scX$ intact but modifies the weight. The first is pulling back the test configuration over a finite map $\comp\to\comp$ given by the map $t\to t^d$, which modifies the total weight $w(k)$ by multiplication by $d$. The second is by adding a constant $\tau$ to the weights, which modifies the total weight by adding $\tau kh(k)$. We use these operations to ensure the test configuration induces one-parameter subgroups of $SL(h(r),\comp)$, following the previous paragraph. Firstly, we pull-back over a finite map $t\to t^{rh(r)}$, so that the total weight is $w(r)rh(r)$. Next we add $-rw(r)$ to the weights, ensuring the test configuration now lies in $SL(h(r),\comp)$.

From the exact sequence (\ref{hilbpoint}), the line above the limit scheme $\scY_0$ is given as $$\Lambda^{max}H^0(\scY_0,\scL_0^{K})\otimes \Lambda^{max}S^KH^0(M,L_1^r)^*.$$ Since the one-parameter subgroup we consider lies in $SL(h(r))$, the induced action on $\Lambda^{max}S^KH^0(X,L^r)^*$ has weight zero. We therefore need only calculate the change in weight for the action on $\Lambda^{max}H^0(\scY_0,\scL_0^{K})$. The finite cover $t\to t^{rh(r)}$ of the test configuration modifies this weight by $\hat{w}(rK)\to \hat{w}(rK)rh(r)$, while adding $-rw(r)$ to the weights modifies the weight by $$\hat{w}(rK)rh(r)\to\hat{w}(rK)rh(r) - (rw(r))K\hat{h}(rK).$$ We conclude by recalling $k=rK$.

The Chow statement follows since for \emph{fixed} $r$, the Chow weight is the leading order term in the polynomial which determines the Hilbert weight \cite[Lemma 2.11]{Mum} \cite[Theorem 3.9]{RT2}.

 \end{proof}
 
 \begin{remark} Theorem \ref{numericalstability} extends the well known criterion for asymptotic Hilbert and Chow stability obtained by setting $M=Y$, see for example \cite[Theorem 3.9]{RT2} whose proof we have followed above. The reason that we have discussed Hilbert stability in this section is primarily its use in the above proof. Indeed, to show the $X$-twisted asymptotic Chow weight is a polynomial, we first showed the $X$-twisted asymptotic Hilbert weight was a polynomial and used that the Chow weight is the leading order term in the Hilbert weight. \end{remark}

We now focus on the setting relevant to the J-flow, which is a minor modification of the above. Here we fix a polarised variety $(M,L_1)$ together with an auxiliary ample line bundle $L_2$. The notion of stability we then define is not a bona fide GIT notion, instead we formally adapt the definition of stability using the Hilbert-Mumford criterion to incorporate the linear system associated to $L_2$. For the moment we assume that $|L_2|$ is an arbitrary linear system.

\begin{lemma} Let $M\subset\pr^n$ be a projective variety together with a linear system $|L_2|$. For each one-parameter subgroup $\lambda\hookrightarrow SL(n+1,\comp)$, the Chow weight of $\lambda$ for $D\in |L_2|$ is constant outside a Zariski closed subset of $|L_2|$. We define the Chow weight of $\lambda$ for $|L_2|$ to equal this general value. \end{lemma}

\begin{proof} The Chow weight is the GIT weight of the one-parameter subgroup in an appropriate Chow variety. Note that the linear system $|L_2|$ is a projective space, hence irreducible, and the image of an irreducible variety is irreducible. Therefore Lemma \ref{generalweight} implies the Chow weight of $\lambda$ for $D\in |L_2|$ is constant outside a Zariski closed subset of $|L_2|$, as required. \end{proof}

\begin{remark} To the authors' knowledge, this phenomenon was first noticed by Sz\'ekelyhidi in the study of twisted K\"ahler-Einstein metrics \cite{Sz-Conical}. \end{remark}

\begin{definition}Let $M\subset\pr^n$ be a projective variety together with a linear system $|L_2|$. We say that $|L_2|$ is \emph{Chow stable} if for each one-parameter subgroup $\lambda\hookrightarrow SL(n+1,\comp)$, the Chow weight of $\lambda$ for $|L_2|$ is strictly positive. \end{definition}

This immediately implies that the \emph{asymptotic} Chow weight of a linear system for a fixed test configuration is constant outside a Zariski closed subset of $|L_2|$. Note that, as the $M$-twisted asymptotic Chow weight for a fixed divisor $D$ is a polynomial in $r$, it is determined by finitely many values, and hence its general value is also constant outside a Zariski closed subset of $|L_2|$.  

\begin{definition} We say that a linear system $|L_2|$ is \emph{asymptotically Chow stable} if for all $r\gg 0$, the asymptotic Chow weight of each test configuration of exponent $r$ is strictly positive. This is equivalent to the linear system being Chow stable in $\pr(H^0(M,L_1^r)^*)$ for $r\gg 0$. \end{definition}

The following definition is equivalent to one due to Lejmi-Sz\'ekelyhidi \cite{L-Sz}. Indeed our definition of Chow stability of a linear system is motivatived by their work.

\begin{definition}\label{jstabilityone} We define the J-weight $J_{L_2}(\scX,\scL)$ of a test configuration $(\scX,\scL)$ to equal $1/a_0$ times the leading order term of the asymptotic Chow weight of the linear system $|L_2|$. We say that $(M,L_1,L_2)$ is \emph{J-stable} (resp. J-semistable) if $J_{L_2}(\scX,\scL)>0$ (resp. $\geq 0$) for all non-trivial test configurations with normal total space. Explicitly, the J-weight of a test configuration is$$J_{L_2}(\scX,\scL) = \frac{\hat{b}_0a_0 - b_0\hat{a}_0}{a_0}.$$ Indeed, one obtains this by simply expanding the polynomial $e_{m+1}(r)$ defined in Theorem \ref{numericalstability}; the J-weight is simply $1/a_0$ times the leading order term in $r$ of this polynomial. \end{definition}

\begin{remark} Our assumption that the total space $\scX$ of the test configuration is normal is to exclude certain pathological test configurations which normalise to the trivial test configuration, and necessarily have J-weight zero. These test configurations are alternatively characterised as having \emph{norm} zero \cite[Theorem 1.3]{Derv2} \cite{BHJ}. \end{remark}

For comparison we state the definition of K-stability.

\begin{definition}\label{kstability} Let $(X,L)$ be a polarised normal variety. We define the Donaldson-Futaki invariant of a test configuration to be (a positive constant times) the leading order term in its asymptotic Chow weight. Explicitly, we define $$\DF(\scX,\scL) = \frac{b_0a_1 - b_1a_0}{a_0}.$$ We say that $(X,L)$ is K-stable if $\DF(\scX,\scL)>0$ for all non-trivial test configurations with normal total space.\end{definition}

\begin{proposition} Assume the linear system $|L_2|$ is asymptotically Chow semistable. Then $(M,L_1,L_2)$ is J-semistable. \end{proposition}

\begin{proof} This is immediate as the $J$-weight of a test configuration is the leading order term in the asymptotic Chow weight for the linear system $|L_2|$. \end{proof}

\subsection{J-balanced metrics and Chow stability}

In this section we work with a smooth projective variety $M$ embedded in a fixed projective space $\pr^n$, and a very ample line bundle $L_2$. By Bertini's Theorem, since $|L_2|$ is basepoint free, a general element $D\in |L_2|$ is smooth. The main result of this section is the following.

\begin{theorem}\label{balancedimpliesstable} If $(M,L_1,L_2)$ admits a J-balanced metric, then $|L_2|$ is Chow stable. \end{theorem}

We conjecture that the two are actually equivalent, however technical issues prevent us from proving this. From Corollary \ref{cor1}, we can relate the existence of solutions of the J-flow to algebraic geometry as follows.

\begin{corollary}\label{corChow} Suppose there exists a fixed point of the J-flow, i.e a critical metric solution of \eqref{critical}. Then $|L_2|$ is asymptotically Chow stable. \end{corollary}

Our proof is reminiscent of the Kempf-Ness Theorem. The existence of a J-balanced embedding for $(M,L_1,L_2)$ is equivalent to the properness of a certain functional on a space of metrics. Taking a one-parameter subgroup as in the definition of Chow stability of a linear system, we show that properness of the this functional along this one-parameter subgroup forces the Chow weight to be strictly positive. 

To prove Theorem \ref{balancedimpliesstable}, we need some preliminary results. 

\begin{theorem}\cite{L-Sz}\label{generalelement} Let $M$ be a smooth projective $n$-dimensional variety together with a very ample line bundle $L_2$. Let $\alpha\in c_1(L_2)$ be a positive $(1,1)$-form. Then there is a smooth signed measure $\mu$ on the projective space $|L_2|$ such that $$\alpha = \int_{D\in|L_2|}[D]d\mu$$ holds in the weak sense, i.e. for all smooth $(n-1,n-1)$-forms $\beta$ we have $$\int_M \alpha\wedge\beta = \int_{D\in|L_2|}\left(\int_D \beta\right)d\mu.$$ \end{theorem}

The above Theorem allows us to replace our smooth $(1,1)$-form $\chi$ with an integral over the linear system $|L_2|$. To work with the relevant functional for a \emph{fixed} divisor, we use a result due to Phong-Sturm \cite{PhongSturm}. In order to state this result, we define the corresponding functional and a certain norm on sections of line bundles over Grassmannians.

\begin{definition} Consider $D\in |L_2|$ a smooth divisor in $M$. We define the functional $I_D^{\text{AYM}}$ over the space of $Met(L_1)$ variationally by $$\frac{d}{dt}I^{\text{AYM}}_D(\phi(t)) = -\frac{1}{\Vol_{{L_1}}(D)} \int_D \dot{\phi}_t \omega_{\phi_t}^{n-1},$$ taking the value zero on $\phi=0$. \\
Here $\Vol_{{L_1}}(D)=\int_D c_1(L_1)^{n-1} = \int_M c_1(L_1)^{n-1}.c_1(L_2)$ is the volume of $D$. \end{definition} 

Comparing this to the functional relevant to the existence of J-balanced metric (see Section \ref{J-func}), defined variationally by \begin{equation}\frac{d}{dt}J_{\chi}(\phi(t)) = \frac{1}{\gamma} \int_M \dot{\phi}(t)\chi\wedge \omega_{\phi}^{n-1},\end{equation} where $\chi\in c_1(L_2)$ is a smooth positive $(1,1)$-form, we obtain the following Corollary of Theorem \ref{generalelement}.

\begin{lemma}\label{Jasaverage} The J-balancing functional satisfies $$J_{\chi}(\phi) = -\Vol_{L_1}(M)\int_{D\in|L_2|} I_D^{\text{AYM}}(\phi)d\mu,$$
where $d\mu$ is chosen as in Theorem \ref{generalelement}.\end{lemma}

Here we may integrate only over the smooth elements $D\in|L_2|$, since the complement has measure zero (by Bertini's Theorem) this does not affect the value of the integral. 

In the previous section, in defining stability of projective varieties, we have fixed a basis of projective space and acted on the variety itself by one-parameter subgroups. An alternative, but equivalent, point of view that we will now take up is to \emph{fix a variety} and \emph{vary the basis} of projective space.

Let $\omega=\omega_{FS}$ be the Fubini-Study metric of projective space. For $\sigma \in SL(n+1,\comp)$, define \begin{equation}\label{projphi}\phi_{\sigma} = \log \left(\frac{|\sigma.x|^2}{|x|^2}\right),\end{equation} and also $$\omega_{\sigma} = \omega + i\partial\bar{\partial} \phi_{\sigma}.$$ In this way we can consider both $J_{\chi}(\phi)$ and $I_D^{\text{AYM}}(\phi)$ as functionals on $SL(n+1,\comp)$. By abuse of notation, we denote these functionals as $J_{\chi}(\sigma)$ and $I_D^{\text{AYM}}(\sigma)$ for $\sigma\in SL(n+1,\comp)$.

We now define a norm on the space of sections over certain Grassmannians, following Phong-Sturm.

\begin{definition}\cite[Section 4]{PhongSturm} Let $f\in H^0(\Grass(n-m,n+1), \scO(d))$ be the Chow point of a degree $d$ variety in $\pr^n$. Denote by $\omega_{Gr}=Pl^*(\omega_{FS})$, where we recall from equation (\ref{plucker}) that $Pl$ is the Pl\"ucker embedding of the Grassmannian. We define the norm of $f$ as $$\log \|f\|^2 = \frac{1}{\Vol(Gr)}\int_{Gr}\log\frac{|f(z)|^2}{|Pl(z)|^2}\omega_{Gr}^r,$$ where $\Vol(Gr)=\int_{Gr}\omega^{m+1}_{Gr}$ is the volume of the Grassmannian and $r+1=(n-m)(m+1)$ is its dimension. \end{definition}

\begin{theorem}\cite[Theorem 5]{PhongSturm} Let $D\subset\pr^n$ be a smooth projective variety of dimension $m$. Denoting by $f$ the Chow point of $Y$, we have, where $V$ is the volume of $D$ with respect to the Fubini-Study metric \begin{equation} -V (m+1)I_D^{\text{AYM}}(\sigma)= \log\frac{\|\sigma.f\|^2}{\|f\|^2}.\end{equation} \end{theorem}

\begin{corollary}\label{propertochoweight} Suppose the functional  $$ -I_D^{\text{AYM}}(\sigma): SL(n+1,\comp)\to\R$$ is proper. Let $\sigma(t)\hookrightarrow SL(n+1,\comp)$ be a one-parameter subgroup. Then the Chow weight of $D$ with respect to $\sigma(t)$ is strictly positive.\end{corollary}

\begin{proof} The properness of the functional means that $\|\sigma.f\|^2\to\infty$ as $\|\sigma\|^2\to\infty$. Applying this to our one-parameter subgroup, this implies the orbit $\sigma_t.f$ is closed.

To relate the closedness of the orbit to the Chow weight, we use the Hilbert-Mumford criterion. Let $f_0$ be the limit $\lim_{t\to\infty}\sigma_t.f$, which is a fixed point of the one-parameter subgroup. Write $\hat{f}_0$ for a lift of $f_0$. Defining $w$ by $\lambda(t).\hat{x}_0 = t^w\hat{x}_0$, the Chow weight is $-w$. If $w>0$, then $0$ lies in the closure of the orbit of the one-parameter subgroup, hence the orbit is not closed. Similarly if $w=0$ then the action is trivial on the line above $x_0$, and so the orbit of a lift of $x$ cannot be closed. We conclude that the orbit being closed implies the Chow weight is positive.\end{proof}

We can now prove Theorem \ref{balancedimpliesstable}.

\begin{proof}[Proof of Theorem \ref{balancedimpliesstable}] By Lemma \ref{Jasaverage}, we have $$J_{\chi}(\phi) = -\Vol_{L_1}(M)\int_{D\in|T|} I_D^{\text{AYM}}(\phi).$$ For $\sigma \in SL(n+1,\comp)$ this implies \begin{equation}\label{functionaloversl}J_{\chi}(\sigma) = -\Vol_{L_1}(M)\int_{D\in|L_2|} I_D^{\text{AYM}}(\sigma).\end{equation} By hypothesis, the functional $J_{\chi}(\varphi)$ is proper. Fix a one-parameter subgroup $\sigma(t) \hookrightarrow SL(n+1,\comp)$. By Lemma \ref{generalweight}, the Chow weight of this one-parameter subgroup is constant outside a Zariski closed subset of $|L_2|$. In the integral over $|L_2|$ in equation (\ref{functionaloversl}), we ignore the union of the non-smooth $D$ together with the subset on which the Chow weight is not equal to the general value. Remark this set has measure zero, and hence this does not affect the value of the integral.  

For the one-parameter subgroup $\sigma(t)$, one of two things must therefore happen. Either the general Chow weight is positive, or the general Chow weight is non-positive. Remark that by Lemma \ref{generalweight}, in the latter case the Chow weight is non-positive for \emph{each} $D\in |L_2|$.  If the general Chow weight is positive for each one-parameter subgroup, then by definition the linear system $|L_2|$ is Chow stable and we are done. We therefore wish to show that if each Chow weight is non-positive, the functional $J_{\chi}(\sigma(t))$ is not proper.

We now consider the latter case. We use the fact that a $C^2$ convex function $f(t):\R\to\R$ is proper if and only if $$\displaystyle\lim_{t\to\infty} \frac{d}{dt} f'(t)>0,$$ where the value infinity is allowed. Applying this in our situation along the one-parameter subgroup $\sigma(t)$, by Corollary \ref{propertochoweight} we see that the Chow weight for each $D\in |L_2|$ being non-positive implies for smooth $D$ that \begin{equation}\label{limitderivativenegative}\displaystyle -\lim_{t\to\infty} \frac{d}{dt} (I_D^{\text{AYM}})'(\sigma(t))\leq 0.\end{equation}

We now relate the limit derivatives of the two functionals.
For \emph{fixed} $t$, we have $$J_{\chi}(\sigma(t)) = -\Vol_{L_1}(M)\int_{D\in|L_2|} I_D^{\text{AYM}}(\sigma(t))d\mu $$ and thus $I_D^{\text{AYM}} (\sigma(t))$ is integrable.
Furthermore, for all $D$, $t\mapsto I_D^{\text{AYM}}(\sigma(t))$ can be differentiated and we can express its derivative using the notation of \eqref{projphi} as
$$-I_D^{\text{AYM}}(\sigma(t))'=\frac{1}{\Vol_L(D)}\int_D \frac{x^*\sigma(t)^*(\zeta^*+\zeta)\sigma(t)x}{\Vert \sigma(t)x\Vert^2}\omega_{\sigma(t)}^{n-1} $$
where $x=(1,z_1,...,z_N)$ and as explained in \cite{PhongSturm}. Here $\sigma(t)=e^{\zeta t}\sigma_0$ with $\zeta$ traceless Hermitian. Consequently, if $r_\zeta$ is the spectral radius of $\zeta$, we have
$$\vert I_D^{\text{AYM}}(\sigma(t))'\vert \leq 2r_\zeta Vol_L(D)$$
and $D\mapsto \Vol_L(D)$ is integrable over $|L_2|$. Hence a variant of Lebesgue dominated convergence theorem asserts that $$\frac{d}{dt} J_{\chi}(\sigma(t)) = -\Vol_{L_1}(M)\int_{D\in|L_2|} \left(\frac{d}{dt} I_D^{\text{AYM}}(\sigma(t))\right)d\mu.$$ 
Now, for almost all $D \in |L_2|$, $\frac{d}{dt} (I_D^{\text{AYM}})'(\sigma(t))$ converges simply when $t$ tends to $+\infty$, so together with above properties, another variant of Lebesgue dominated convergence theorem gives
\begin{align*} \lim_{t\to+\infty}\frac{d}{dt} J_{\chi}(\sigma(t)) &= -\lim_{t\to\infty}\Vol_{L_1}(M)\int_{D\in|L_2|} \left(\frac{d}{dt} I_D^{\text{AYM}}(\sigma(t))\right)d\mu, \\ &= -\Vol_{L_1}(M)\int_{D\in|L_2|} \left(\lim_{t\to+\infty}\frac{d}{dt} I_D^{\text{AYM}}(\sigma(t))\right)d\mu.\end{align*} Equation (\ref{limitderivativenegative}) then implies $$ \lim_{t\to\infty}\frac{d}{dt} J_{\chi}(\sigma(t)) \leq 0,$$ contradicting properness. This completes the proof. \end{proof}

\subsection{The J-flow and K-stability}\label{jflowsection}

In this section we study J-stability, as in Definition \ref{jstabilityone}. In this definition one takes a projective variety $M$ with two ample line bundles $L_1,L_2$, and associates a weight called the J-weight $J_{L_2}(\scX,\scL)$ to each test configuration $(\scX,\scL)$. J-stability then requires that $J_{L_2}(\scX,\scL)>0$ for each non-trivial test configuration with $\scX$ normal. In order to study J-stability, we reduce the class of test configurations needed to those which are blow-ups along flag ideals. 

\begin{remark} To clarify notation in this section, for a line bundle $L$ we denote $rL$ its $r$th tensor power. We denote intersection numbers as $L^n=\int_M c_1(L)^n$ and so on. We also use additive notation for the tensor product of line bundles. \end{remark}

\begin{definition} A flag ideal $\scI$ is a coherent ideal sheaf on $M\times\comp$ of the form $\scI = I_0+(t)I_1+\hdots +(t^N)$, where $t$ is the coordinate on $\comp$ and $I_0\subset\hdots\subset I_{N-1}\subset\scO_M$ are a sequence of coherent ideal sheaves on $M$ corresponding to subschemes $Z_0\supset Z_1 \supset \hdots \supset Z_{N-1}$ of $M$. \end{definition}

Denote by $\tilde{\scB}$ the blow-up of $\scI$ on $X\times\comp$, i.e. $$\pi: \tilde{\scB} = \Bl_{\scI} M\times\comp\to M\times\comp,$$ denote by $\scL_1,\scL_2$ the pullbacks of $L_1,L_1$ to $\tilde{\scB}$ and let $\scO(-E)=\pi^{-1}\scI$ be the exceptional divisor of the blow-up. The map $\tilde{\scB}\to\comp$ is flat, and the natural $\comp^*$-action on $X\times\comp$ lifts to $\tilde{\scB}$. It follows that $(\tilde{\scB},r\scL_1-E)$ is a test configuration of exponent $r$ for $(X,L_1)$ provided $r\scL_1-E$ is relatively ample. Remark that each such test configuration has a canonical compactification obtained by blowing up $\scI$ on $M\times\pr^1$; we denote this test configuration by $\scB$ and by abuse of notation denote $\scL_1,\scL_2,E$ the corresponding line bundles and divisors on $\scB$. The following Proposition then states that it is enough to check J-stability with respect to these test configurations, provided one allows $r\scL_1-E$ to be relatively \emph{semi}-ample. 

\begin{proposition}\label{jflowblow-ups} Let $M$ be a normal projective variety with ample line bundles $L_1,L_2$. Then $(M,L_1,L_2)$ is J-stable is equivalent to the fact that  $$J_{L_2}(\scB,r\scL_1-E)>0$$ for all flag ideals $\scI$ with $\scB=\Bl_{\scI}M\times\pr^1$ normal and $r\scL-E$ relatively semi-ample over $\pr^1$. Recall the J-constant of $(M,L_1,L_2)$ is defined as  $$\gamma = \frac{L_2.L_1^{n-1}}{L_1^n}.$$ On such blow-ups the  J-weight defined by Definition \ref{jstabilityone} is given by the formula \begin{equation}\label{numericaljweight}J_{L_2}(\scB,r\scL_1-E) = (r\scL_1-E)^n.\left(-\frac{n}{n+1}\gamma r^{-1}(r\scL_1-E) + \scL_2\right),\end{equation} up to multiplication by a positive dimensional constant.\end{proposition}

\begin{proof} From \cite[Corollary 3.11]{Odaka2}, given an arbitrary test configuration $(\scX,\scH)$, there exists a flag ideal $\scI$ and a $\comp^*$-equivariant map $$\psi: (\tilde{\scB},r\scL_1-E)\to (\scX,\scH)$$ with $\tilde{\scB}$ normal and $r\scL_1-E=\psi^*\scH$ relatively semi-ample over $\comp$. Since the map $\psi$ is $\comp^*$-equivariant, the total weight $\wt H^0(\scX_0,k\scH_0)$ is equal to the total weight $\wt H^0(\tilde{\scB}_0,k(r\scL_1-E)_0)$. Similarly given a divisor $D\subset M$, denote by $\scD\subset \scX$ the closure under the $\comp^*$-action. If one denotes by $\tilde{\scB}_D$ the proper transform of $D\times\comp$, then similarly by \cite[Proposition 3.5]{Od-Su} we have $$\wt H^0(\scD_0,k\scH_0)=\wt H^0(\tilde{\scB}_{D,0},k(r\scL_1-E)|_0).$$

To obtain a formula for the J-weight, we need to calculate the $b_0$ and $\hat{b}_0$, where the latter is for a general divisor $D \in |T|$. We first of all compactify the above semi-test configuration as detailed above. We denote by $\scB$ and $\scB_D$ the corresponding test configurations over $\pr^1$. For the $b_0$ term, by \cite[Theorem 3.2]{Odaka2} or \cite[Proposition 17]{Wang4} we have \begin{align*}\wt(H^0(\scB_0,k(r\scL_1-E)|^k_0)) &= \chi(\scB,k(r\scL_1-E))+O(k^{n-1}), \\ &= \frac{(r\scL_1-E)^{n+1}}{(n+1)!}k^{n+1}+O(k^n),\end{align*} using asymptotic Riemann-Roch for normal varieties \cite[Lemma 3.5]{Odaka2}. Similarly for arbitrary $D\in |L_2|$ we have \begin{align*}\wt(H^0(\scB_{D,0},k(r\scL_1-E)|_0)) &= \chi(\scB_D,k(r\scL_1-E)) +O(k^{n-2}), \\ &=\frac{(r\scL_1-E)^n.\scB_{D}}{n!}k^n+O(k^{n-1}), \\ &=  \frac{(r\scL_1-E)^n.(\scL_2+(\scB_{D}-\scL_2))}{n!}k^n+O(k^{n-1}).\end{align*}  

Here $\scL_2$ is the line bundle associated to the \emph{total} transform of $D\times\pr^1$, while $\scB_D$ corresponds to the \emph{proper} transform. Remark that these are equal for general $D$, indeed they are not equal if and only if the flag $\scI$ has a component contained in $D$. It follows that $$\hat{b}_0 = \frac{(r\scL_1-E)^n.\scL_2}{n!}.$$ 

Summing up, we see that $$(n+1)! J_{L_2}(\scB,r\scL_1-E) =  (r\scL_1-E)^n.\left(-\frac{n}{n+1}\gamma r^{-1}(r\scL_1-E) + \scL_2\right),$$ using $$\gamma = nr\frac{\hat{a}_0}{a_0}.$$

\end{proof}

Using the above, we can extend the definition of J-stability to the case where $L_2$ is an arbitrary line bundle, which will be useful in applications. By Proposition \ref{jflowblow-ups}, it follows that when $L_2$ is ample, this definition is equivalent to Definition \ref{jstabilityone}.

\begin{definition}\label{jstabledefinitiongeneral} Let $(M,L_1,L_2)$ be a normal projective variety $M$ with an ample line bundle $L_1$ and an auxiliary line bundle $L_2$, not necessarily ample. We say that $(M,L_1,L_2)$ is \emph{J-stable} if for each flag ideal $\scI$, the corresponding J-weight given in equation (\ref{numericaljweight}) is strictly positive. \end{definition}

In order to use this formalism, we need the following positivity properties of certain intersection numbers.

\begin{lemma}\label{inequalities}\cite[Proposition 4.3, Theorem 2.6]{Od-Sa} \cite[Equation (3)]{Odaka3} \cite[Lemma 3.7]{Derv} \\ 
With all notation as above, let $R$ be a nef divisor on $M$, and denote $p^*R = \scR$ where $p: \scB\to M$ is the natural morphism induced by the blow-up map. Then the following positivity properties of intersection numbers hold. \begin{itemize} \item[(i)] $(r\scL_1-E)^n.\scR \leq 0$, \item[(ii)] $(r\scL_1-E)^n.E > 0$,  \item[(iii)] $(r\scL_1-E)^n.(r\scL_1+nE)>0$. \end{itemize}\end{lemma}

We can now apply this blowing-up formalism.

\begin{theorem}\label{sufficientconditionjstability} Suppose that $\gamma L_1 - L_2$ is nef, with $\gamma>0$. Then $(M,L_1,L_2)$ is J-stable. \end{theorem}

\begin{proof} We use the blowing-up formalism of Proposition \ref{jflowblow-ups}. Let $\scI$ be a flag ideal with corresponding blow-up $\scB$. The J-weight is given as \begin{align*}J_{L_2}(\scB,r\scL_1-E) &= (r\scL_1-E)^n.\left(-\frac{n}{n+1}\gamma r^{-1}(r\scL_1-E) + \scL_2\right), \\ 
&= (r\scL_1-E)^n.\left(\frac{1}{n+1}\gamma r^{-1}(r\scL_1+nE)+(-\gamma \scL_1 + \scL_2)\right). 
\end{align*} 
By Lemma \ref{inequalities} $(iii)$, we have $$ (r\scL_1-E)^n.(r\scL_1+nE)>0,$$ while Lemma \ref{inequalities} $(i)$ together with the assumption that $\gamma L_1 - L_2$ is nef gives $$(r\scL_1-E)^n.(-\gamma \scL_1 + \scL_2) \geq 0.$$ Combining these proves the result. \end{proof}

\begin{remark} Since $\gamma = \frac{L_1^{n-1}.L_2}{L_1^n}$, the assumption $\gamma>0$ is automatic when $L_2$ is ample, or even effective. The above result holds for general $L_2$ however using Definition \ref{jstabledefinitiongeneral}, provided $\gamma>0$.\end{remark}

\begin{remark} In the case $L_2$ is ample, a result of Weinkove \cite{We1} states that if $$\gamma L_1 - \frac{n-1}{n}L_2$$ is also ample, then the $I_{\mu_J}$ functional is bounded. Note that this is a weaker assumption than we made in Theorem \ref{sufficientconditionjstability}; however we do not assume $L_2$ is ample. It would be interesting to directly prove J-stability under Weinkove's hypothesis. \end{remark}

We now prove a link between J-stability and K-stability (as in Definition \ref{kstability}). This is the algebro-geometric analogue of the relationship between existence of solutions to the J-flow and coercivity of the Mabuchi functional, due to Chen \cite{C2}. Our result will make use of certain measures of singularities of projective varieties.

\begin{definition}\label{lc} Let $M$ be a normal $\Q$-Gorenstein variety. Let $\pi: Y\to M$ be an arbitrary birational map with $Y$ normal. We can then write \begin{equation*}K_Y - \pi^*K_M \equiv \sum a(E_i,M)E_i. \end{equation*}We say that $M$ is \emph{Kawamata log terminal} if $a(E_i,M) > -1$ for all $E_i$ with $M$ normal. By \cite[Lemma 3.13]{JaKo} it suffices to check this property for $Y\to M$ a resolution of singularities. In particular, smooth varieties are Kawamata log terminal. \end{definition}

\begin{theorem}\label{jimpliesk} Suppose $(M,L_1,K_M)$ is J-semistable, with $M$ Kawamata log terminal. Then $(M,L_1)$ is K-stable. \end{theorem}

\begin{proof}A result of Odaka \cite[Corollary 3.11]{Odaka2} states that, similarly to Proposition \ref{jflowblow-ups}, to check K-stability it suffices to show the Donaldson-Futaki invariant of each semi-test configuration given in Proposition \ref{jflowblow-ups} is strictly positive. The proof is a comparison of the Donaldson-Futaki invariant and J-weight. Indeed, letting $\scB$ be a blow-up along a flag ideal as above, we have $$J_{K_M}(\scB,r\scL_1-E) = (r\scL_1-E)^n.\left(-\frac{n}{n+1}\gamma r^{-1}(r\scL_1-E) + \scK_M\right),$$ while the corresponding Donaldson-Futaki invariant is given by Odaka as $$ \DF(\scB,r\scL_1-E)\hspace{-0.1cm}=\hspace{-0.1cm}(r\scL_1-E)^{n}.\hspace{-0.05cm}\left(-\frac{n}{n+1}\gamma r^{-1}(r\scL_1-E)\hspace{-0.05cm}+\hspace{-0.05cm}\scK_M\hspace{-0.05cm}+\hspace{-0.05cm} K_{\scB/M\times\pr^1}\hspace{-0.05cm}\right).$$ Here we have denoted $\scK_M$ the pullback of $K_M$ to $\scB$. Hence $$ \DF(\scB,r\scL_1-E) = J_{K_M}(\scB,\scL_1^r-E) + (r\scL_1-E)^n.K_{\scB/M\times\pr^1}.$$ 

The term $K_{\scB/M\times\pr^1}=K_{\scB}-\pi^*{K_{M\times\pr^1}}$ is an exceptional divisor of the blow-up $\pi: \scB\to M\times\pr^1$, since $\scB$ is normal the intersection numbers make sense. Hence to show K-stability, by Lemma \ref{inequalities} $(ii)$, it suffices to show that $K_{\scB/M\times\pr^1}$ is effective. This follows by inversion of adjunction \cite[Theorem 5.50]{KM}. Indeed, $M$ being Kawamata log terminal implies that $X\times\pr^1$ is purely log terminal, hence the discrepancy term is effective. 
\end{proof}

The following Corollary is our motivation for extending the Definition of J-stability to the case with $L_2$ not necessarily ample.

\begin{corollary}\label{kstablecone} Suppose that $(M,L_1)$ is a Kawamata log terminal variety with $\gamma>0$ which satisfies $$\gamma L_1 - K_M \geq 0,$$ i.e. the difference is nef. Then $(M,L_1)$ is K-stable. \end{corollary}

\begin{proof} This is immediate by combining Theorems \ref{sufficientconditionjstability} and \ref{jimpliesk}.\end{proof}

\begin{remark}This Corollary should be compared to Weinkove's work \cite{We1} in the smooth case with $K_M$ ample, which proves that the Mabuchi functional for $(M,L_1)$ is proper provided $$\gamma L_1 - \frac{n-1}{n}K_M$$ is ample. While we need a stronger ampleness criterion for $L_1$, we do \emph{not} need to assume $K_M$ is ample, merely that $\gamma>0$.\end{remark}

\begin{remark} Corollary \ref{kstablecone} was also obtained by the first author through a direct analysis of the K-stability condition \cite[Theorem 1.7]{Derv2}. When $L_1=K_M$, i.e. $M$ is a canonically polarised variety, this result is due to Odaka \cite{Odaka3}.\end{remark}

When $M$ is a surface we can improve these results, using another intersection theoretic Lemma. This is essentially a strengthening of Lemma \ref{inequalities} $(iii)$ for surfaces.

\begin{lemma}\label{surfacesineq} Suppose $M$ has dimension $2$. Then $$(r\scL_1-E)^2.(r\scL_1+E)\geq 0.$$ \end{lemma}

\begin{proof} The intersection number expands as $$(r\scL_1-E)^2.(r\scL_1+E)= r^3\scL_1^3-r^2\scL_1^2.E-r\scL_1.E^2+E^3.$$ As $\scL$ is the pullback of a line bundle from the surface $M$, we have $\scL_1^3=0$. 

Recall that the flag ideal is of the form $\scI = I_0+(t)I_1+\hdots +(t^N)$. As in the proof of \cite[Theorem 2.6]{Odaka3}, we can assume that the flag ideal has support $$s= \dim \Supp (\scO_{M\times\pr^2} / \scI)\leq 1.$$ Indeed otherwise, the flag idea satisfies $I_0=\scO_M$ and dividing by a power of $t$ does not change the blow-up and hence the Donaldson-Futaki invariant, but ensures $s\leq 1$. Then provided $s\leq 1$, we have $\scL_1^2.E=0$ by the projection formula, as in \cite[Lemma 3.5]{Odaka1}.

It follows that \begin{align}(r\scL_1-E)^2.(r\scL_1+E) &= -r\scL_1.E^2+E^3, \\ &= -E^2.(r\scL_1-E).\end{align} But this latter term is non-negative by \cite[Lemma 2.8 (i)]{Odaka3} (setting $i=1$ in the notation of that Lemma), concluding the proof.\end{proof}

Using the above Lemma we can strengthen Theorem \ref{sufficientconditionjstability} for surfaces.

\begin{theorem}\label{jstablesurfaces}Suppose $M$ has dimension $2$ satisfying $\frac{4}{3}\gamma L_1 - L_2$ is nef, and $\gamma>0$. Then $(M,L_1,L_2)$ is J-semistable. \end{theorem}

\begin{proof} We follow the proof of Theorem \ref{sufficientconditionjstability} using Lemma \ref{surfacesineq}. Letting $\scI$ be a flag ideal as above, the J-weight is given as \begin{align*}J_{L_2}(\scB,r\scL_1-E) &= (r\scL_1-E)^2.\left(-\frac{2}{3}\gamma r^{-1}(r\scL_1-E) + \scL_2\right), \\ 
&= (r\scL_1-E)^2.\left(\frac{2}{3}\gamma r^{-1}(r\scL_1+E)+ (-\frac{4}{3}\gamma \scL_1 + \scL_2) \right)
\end{align*}

By Lemma \ref{surfacesineq} we have $$(r\scL_1-E)^2.(r\scL_1+E)\geq 0,$$ while Lemma \ref{inequalities} $(i)$ together with the hypothesis of the Theorem ensures $$(r\scL_1-E)^2.\left(-\frac{4}{3}\gamma \scL_1 + \scL_2\right)\geq 0.$$

\end{proof}

\begin{remark} It is important to note that in the above result we only prove J-\emph{semi}stability, it would be interesting to prove J-stability assuming the line bundle $\frac{4}{3}\gamma L_1 - L_2$ is actually ample. This does not follow directly from our method. \end{remark}

We immediately obtain the following Corollary, by Theorem \ref{jimpliesk}. This is the most general currently known result for K-stability of surfaces of general type. 

\begin{corollary}\label{surfaces} Let $(M,L_1)$ be a polarised Kawamata log terminal surface satisfying $\frac{4}{3}\gamma L_1 - K_M$ is nef and $\gamma>0$. Then $(M,L_1)$ is K-stable. \end{corollary}

\begin{remark} Panov and Ross have proved slope stability of surfaces with ample canonical class under the weaker assumption $2\gamma L_1 - K_M$ is ample \cite[Example 5.8]{Panov-Ross}. However by \cite[Example 7.8]{Panov-Ross}, slope stability is a strictly weaker condition than K-stability, i.e. K-stability implies slope stability but there are examples of slope stable polarised varieties which are not K-stable.
\end{remark}

Corollaries \ref{surfaces} and \ref{kstablecone}, together with the link between J-stability and the $I_{\mu_J}$ functional lead us to the following conjecture. We remark that the definition of the $I_{\mu_J}$ functional makes sense with $L_2$ arbitrary, taking a not necessarily positive $(1,1)$-form $\chi\in c_1(L_2)$.

\begin{conjecture}\label{j-functionalwithoutampleness} Let $M$ be a smooth $n$-dimensional variety with an ample line bundle $L_1$ and an arbitrary line bundle $L_2$. Suppose $\gamma>0$ and $$\gamma L_1 - \frac{n-1}{n}L_2$$ is ample. Then the $I_{\mu_J}$ functional is proper. In particular if $L_2=K_M$, then $(X,L_1)$ has proper Mabuchi functional. \end{conjecture}

We emphasise again that in the above Conjecture we do \emph{not} assume $L_2$ is positive just that $\gamma>0$. In the case $L_2$ is positive the above Conjecture follows from Weinkove's work \cite{We1}. In the case of minimal surface of general type (i.e with big and nef canonical bundle), the Conjecture is proved in \cite{SW2}.

\bibliography{Jflow-CAG2017.bib}

\begin{thebibliography}{Oda13b}

\bibitem[Ber09]{Bern}
B.~Berndtsson.
\newblock Positivity of direct image bundles and convexity on the space of
  {K}\"ahler metrics.
\newblock {\em J. Differential Geom.}, 81(3):457--482, 2009.

\bibitem[BHJ15]{BHJ}
S.~{Boucksom}, T.~{Hisamoto}, and M.~{Jonsson}.
\newblock {Uniform {K}-stability, {D}uistermaat-{H}eckman measures and
  singularities of pairs}.
\newblock {\em ArXiv e-prints math.AG/1504.06568}, 2015.

\bibitem[BK12]{BerKel}
R.~Berman and J.~Keller.
\newblock Bergman geodesics.
\newblock In {\em Complex {M}onge-{A}mp\`ere equations and geodesics in the
  space of {K}\"ahler metrics}, volume 2038 of {\em Lecture Notes in Math.},
  pages 283--302. Springer, Heidelberg, 2012.

\bibitem[Bou90]{Bou}
T.~Bouche.
\newblock Convergence de la m\'etrique de {F}ubini-{S}tudy d'un fibr\'e
  lin\'eaire positif.
\newblock {\em Ann. Inst. Fourier (Grenoble)}, 40(1):117--130, 1990.

\bibitem[Cat99]{Ca}
D.~Catlin.
\newblock The {B}ergman kernel and a theorem of {T}ian.
\newblock In {\em Analysis and geom. in several complex var. (Katata, 1997)},
  Trends Math. Birkh\"auser, 1999.

\bibitem[Che00]{C2}
X.~X. Chen.
\newblock On the lower bound of the {M}abuchi energy and its application.
\newblock {\em Internat. Math. Res. Notices}, (12):607--623, 2000.

\bibitem[Che04]{C1}
X.~X. Chen.
\newblock A new parabolic flow in {K}\"ahler manifolds.
\newblock {\em Comm. Anal. Geom.}, 12(4):837--852, 2004.

\bibitem[Che15]{C3}
X.~X. Chen.
\newblock On the existence of constant scalar curvature {K}\" ahler metric: a
  new perspective.
\newblock {\em ArXiv e-prints math.DG/1506.06423}, 2015.

\bibitem[CK13]{Ke2}
H.-D. Cao and J.~Keller.
\newblock On the {C}alabi problem: a finite-dimensional approach.
\newblock {\em J. Eur. Math. Soc. (JEMS)}, 15(3):1033--1065, 2013.

\bibitem[CS14]{Sz-Co}
T.~{C}ollins and G.~Sz{\'e}kelyhidi.
\newblock Convergence of the {J}-flow on toric manifolds.
\newblock {\em ArXiv e-prints math.DG/1412.4809, to appear in J. Differential
  Geom.}, 2014.

\bibitem[Der15]{Derv}
R.~Dervan.
\newblock Alpha invariants and {K}-stability for general polarizations of
  {F}ano varieties.
\newblock {\em Int. Math. Res. Not. IMRN}, (16):7162--7189, 2015.

\bibitem[Der16a]{Derv3}
R.~Dervan.
\newblock Alpha invariants and coercivity of the {M}abuchi functional on {F}ano
  manifolds.
\newblock {\em Ann. Fac. Sci. Toulouse Math. (6)}, 25(4):919--934, 2016.

\bibitem[Der16b]{Derv2}
R.~Dervan.
\newblock Uniform stability of twisted constant scalar curvature {K}\"ahler
  metrics.
\newblock {\em Int. Math. Res. Not. IMRN}, (15):4728--4783, 2016.

\bibitem[Don99]{D9}
S.~K. Donaldson.
\newblock Moment maps and diffeomorphisms.
\newblock {\em Asian J. Math.}, 3(1):1--15, 1999.
\newblock Sir Michael Atiyah: a great mathematician of the twentieth century.

\bibitem[Don01]{D1}
S.~K. Donaldson.
\newblock Scalar curvature and projective embeddings. {I}.
\newblock {\em J. Differential Geom.}, 59(3):479--522, 2001.

\bibitem[Don05]{D2}
S.~K. Donaldson.
\newblock Scalar curvature and projective embeddings. {II}.
\newblock {\em Q. J. Math.}, 56(3):345--356, 2005.

\bibitem[Fin10]{Fi1}
J.~Fine.
\newblock Calabi flow and projective embeddings.
\newblock {\em J. Differential Geom.}, 84(3):489--523, 2010.
\newblock With an appendix by Kefeng Liu and Xiaonan Ma.

\bibitem[Gau]{Gaud}
P.~Gauduchon.
\newblock Calabi's extremal {K}\"ahler metrics: {A}n elementary introduction.
\newblock {\em Preprint}, 316 pages.

\bibitem[GRS13]{GRS}
V.~Georgoulas, J.~W. Robbin, and D.~A. Salamon.
\newblock The moment-weight inequality and the {H}ilbert-{M}umford criterion.
\newblock {\em ArXiv e-prints math.SG/1311.0410}, 2013.

\bibitem[KM98]{KM}
J.~Koll{\'a}r and S.~Mori.
\newblock {\em Birational geometry of algebraic varieties}, volume 134 of {\em
  Cambridge Tracts in Mathematics}.
\newblock Cambridge University Press, Cambridge, 1998.
\newblock With the collaboration of C. H. Clemens and A. Corti, Translated from
  the 1998 Japanese original.

\bibitem[Kol96]{JK-rational}
J.~Koll\'ar.
\newblock {\em Rational curves on algebraic varieties}, volume~32 of {\em
  Ergebnisse der Mathematik und ihrer Grenzgebiete. 3. Folge. A Series of
  Modern Surveys in Mathematics [Results in Mathematics and Related Areas. 3rd
  Series. A Series of Modern Surveys in Mathematics]}.
\newblock Springer-Verlag, Berlin, 1996.

\bibitem[Kol97]{JaKo}
J.~Koll{\'a}r.
\newblock Singularities of pairs.
\newblock In {\em Algebraic geometry---{S}anta {C}ruz 1995}, volume~62 of {\em
  Proc. Sympos. Pure Math.}, pages 221--287. Amer. Math. Soc., Providence, RI,
  1997.

\bibitem[LM07]{L-M}
K.~Liu and X.~Ma.
\newblock A remark on: ``{S}ome numerical results in complex differential
  geometry'' [arxiv.org/abs/math/ 0512625] by {S}. {K}. {D}onaldson.
\newblock {\em Math. Res. Lett.}, 14(2):165--171, 2007.

\bibitem[LS15]{L-Sz}
M.~Lejmi and G.~Sz\'ekelyhidi.
\newblock The {J}-flow and stability.
\newblock {\em Adv. Math.}, 274:404--431, 2015.

\bibitem[LSY15]{Li-Shi-Yao}
H.~Li, Y.~Shi, and Y.~Yao.
\newblock A criterion for the properness of the {K}-energy in a general
  {K}{\"a}hler class.
\newblock {\em Math. Annalen}, 361(1-2):135--156, 2015.

\bibitem[Mab86]{Ma1}
T.~Mabuchi.
\newblock {$K$}-energy maps integrating {F}utaki invariants.
\newblock {\em Tohoku Math. J. (2)}, 38(4):575--593, 1986.

\bibitem[MiR00]{MR}
I.~Mundet~i Riera.
\newblock A {H}itchin-{K}obayashi correspondence for {K}\"ahler fibrations.
\newblock {\em J. Reine Angew. Math.}, 528:41--80, 2000.

\bibitem[MM12]{M-M2}
X.~Ma and G.~Marinescu.
\newblock Berezin-{T}oeplitz quantization on {K}\"ahler manifolds.
\newblock {\em J. Reine Angew. Math.}, 662:1--56, 2012.

\bibitem[Mum66]{DM-lectures}
D.~Mumford.
\newblock {\em Lectures on curves on an algebraic surface}.
\newblock With a section by G. M. Bergman. Annals of Mathematics Studies, No.
  59. Princeton University Press, Princeton, N.J., 1966.

\bibitem[Mum77]{Mum}
D.~Mumford.
\newblock {\em Stability of projective varieties}.
\newblock L'Enseignement Math\'ematique, Geneva, 1977.
\newblock Lectures given at the ``Institut des Hautes {\'E}tudes
  Scientifiques'', Monographie de l'Enseignement Math{\'e}matique, No. 24.

\bibitem[Oda12]{Odaka3}
Y.~Odaka.
\newblock The {C}alabi conjecture and {K}-stability.
\newblock {\em Int. Math. Res. Not. IMRN}, (10):2272--2288, 2012.

\bibitem[Oda13a]{Odaka2}
Y.~Odaka.
\newblock A generalization of the {R}oss-{T}homas slope theory.
\newblock {\em Osaka J. Math.}, 50(1):171--185, 2013.

\bibitem[Oda13b]{Odaka1}
Y.~Odaka.
\newblock The {GIT} stability of polarized varieties via discrepancy.
\newblock {\em Ann. of Math. (2)}, 177(2):645--661, 2013.

\bibitem[OS12]{Od-Sa}
Y.~Odaka and Y.~Sano.
\newblock Alpha invariant and {K}-stability of {$\mathbb{Q}$}-{F}ano varieties.
\newblock {\em Adv. Math.}, 229(5):2818--2834, 2012.

\bibitem[OS15]{Od-Su}
Y.~Odaka and S.~Sun.
\newblock Testing log {K}-stability by blowing up formalism.
\newblock {\em Ann. Fac. Sci. Toulouse Math. (6)}, 24(3):505--522, 2015.

\bibitem[PR09]{Panov-Ross}
D.~Panov and J.~Ross.
\newblock Slope stability and exceptional divisors of high genus.
\newblock {\em Math. Ann.}, 343(1):79--101, 2009.

\bibitem[PS03]{PhongSturm}
D.~H. Phong and J.~Sturm.
\newblock Stability, energy functionals, and {K}\"ahler-{E}instein metrics.
\newblock {\em Comm. Anal. Geom.}, 11(3):565--597, 2003.

\bibitem[PS06]{P-S1}
D.~H. Phong and J.~Sturm.
\newblock The {M}onge-{A}mp\`ere operator and geodesics in the space of
  {K}\"ahler potentials.
\newblock {\em Invent. Math.}, 166(1):125--149, 2006.

\bibitem[Ros06]{Ross}
J.~Ross.
\newblock Unstable products of smooth curves.
\newblock {\em Invent. Math.}, 165(1):153--162, 2006.

\bibitem[RT07]{RT2}
J.~Ross and R.~Thomas.
\newblock A study of the {H}ilbert-{M}umford criterion for the stability of
  projective varieties.
\newblock {\em J. Algebraic Geom.}, 16(2):201--255, 2007.

\bibitem[San06]{Sa}
Y.~Sano.
\newblock Numerical algorithm for finding balanced metrics.
\newblock {\em Osaka J. Math.}, 43(3):679--688, 2006.

\bibitem[Sey09]{Sey2}
R.~Seyyedali.
\newblock Numerical algorithm for finding balanced metrics on vector bundles.
\newblock {\em Asian J. Math.}, 13(3):311--321, 2009.

\bibitem[SW08]{SW}
J.~Song and B.~Weinkove.
\newblock On the convergence and singularities of the {$J$}-flow with
  applications to the {M}abuchi energy.
\newblock {\em Comm. Pure Appl. Math.}, 61(2):210--229, 2008.

\bibitem[SW13]{SW2}
J.~Song and B.~Weinkove.
\newblock The degenerate {J}-flow and the {M}abuchi energy on minimal surfaces
  of general type.
\newblock {\em Univ. Iagel. Acta Math.}, (50, [2012 on articles]):89--106,
  2013.

\bibitem[Sz{\'e}13]{Sz-Conical}
G.~Sz{\'e}kelyhidi.
\newblock A remark on conical {K}\"ahler-{E}instein metrics.
\newblock {\em Math. Res. Lett.}, 20(3):581--590, 2013.

\bibitem[Tho06]{ThNotes}
R.~P. Thomas.
\newblock Notes on {GIT} and symplectic reduction for bundles and varieties.
\newblock In {\em Surveys in differential geometry. {V}ol. {X}}, volume~10 of
  {\em Surv. Differ. Geom.}, pages 221--273. Int. Press, Somerville, MA, 2006.

\bibitem[Tia90]{Ti1}
G.~Tian.
\newblock On a set of polarized {K}\"ahler metrics on algebraic manifolds.
\newblock {\em J. Differential Geom.}, 32(1):99--130, 1990.

\bibitem[Wan12]{Wang4}
X.~Wang.
\newblock Height and {GIT} weight.
\newblock {\em Math. Res. Lett.}, 19(4):909--926, 2012.

\bibitem[Wei04]{We2}
B.~Weinkove.
\newblock Convergence of the {$J$}-flow on {K}\"ahler surfaces.
\newblock {\em Comm. Anal. Geom.}, 12(4):949--965, 2004.

\bibitem[Wei06]{We1}
B.~Weinkove.
\newblock On the {$J$}-flow in higher dimensions and the lower boundedness of
  the {M}abuchi energy.
\newblock {\em J. Differential Geom.}, 73(2):351--358, 2006.

\bibitem[Zel98]{Ze}
S.~Zelditch.
\newblock Asymptotics of holomorphic sections of powers of a positive line
  bundle.
\newblock In {\em S\'eminaire sur les \'Equations aux D\'eriv\'ees Partielles,
  1997--1998}, pages Exp. No. XXII, 12. \'Ecole Polytech., Palaiseau, 1998.

\end{thebibliography}

\bigskip

\end{document}